\documentclass{amsart}

\usepackage{bbm}

\usepackage{amsfonts,amsmath, amssymb,amsthm,amscd,stmaryrd,mathtools}
\usepackage[left=3cm,right=3cm,top=3cm,bottom=3cm]{geometry}

\usepackage{thmtools}
\usepackage{verbatim}

\usepackage{float}

\usepackage[bookmarks, colorlinks, breaklinks, pdftitle={},pdfauthor={}]{hyperref}
\hypersetup{
  linkcolor=black,
  citecolor=black,
  filecolor=black,
  urlcolor=black}

\newtheorem{theorem}{Theorem}[section]
\newtheorem{definition}[theorem]{Definition}
\newtheorem{remark}{Remark}

\newtheorem{lemma}[theorem]{Lemma}
\newtheorem{proposition}[theorem]{Proposition}
\newtheorem{corollary}[theorem]{Corollary}

\usepackage{cleveref}

\usepackage{tikz}
\usepackage[font=small,labelfont=bf]{caption}

\usetikzlibrary{calc,decorations.markings,decorations.pathreplacing}
\usetikzlibrary{decorations.pathmorphing,arrows.meta}

\newcommand{\nwithzero}{\N}
\newcommand{\nwithoutzero}{\N_+}

\newcommand{\bb}[1]{\mathbb{#1}}
\newcommand{\cc}[1]{\mathcal{#1}}

\newcommand{\ob}[1]{\left(#1\right )} 
\newcommand{\cb}[1]{\left[#1\right ]} 
\newcommand{\set}[1]{\left\{#1\right\}} 
\newcommand{\abs}[1]{\left\vert#1\right\vert} 
\newcommand{\norm}[1]{\|#1\|} 
\newcommand{\ccb}[1]{\llbracket #1 \rrbracket} 

\newcommand{\indicthat}[1]{{\mathbbm{1}}_{\left\{#1\right\}}}

\newcommand{\E}{\bb E}
\newcommand{\R}{\bb R}
\newcommand{\Z}{\bb Z}

\newcommand{\N}{\bb N}
\renewcommand{\P}{\bb P}

\newcommand{\Var}{\mathrm{Var}}
\newcommand{\T}{\bb T} 
\newcommand{\floor}[1]{\lfloor #1 \rfloor}
\newcommand{\ceil}[1]{\lceil #1 \rceil}
\newcommand{\V}[1]{\boldsymbol{#1}} 
\newcommand{\dist}{d} 
\newcommand{\st}[1]{\mathsf{St}(#1)} 
\newcommand{\sti}[1]{\tilde{\mathsf{St}}(#1)}
\newcommand{\gt}[1]{\mathsf{Gt}(#1)} 
\newcommand{\gtt}[0]{\mathsf{Gt}} 
\newcommand{\cycle}[0]{\mathsf{C}} 
\newcommand{\cstructure}[0]{\cc X} 
\newcommand{\cycles}[0]{ V_{\cycle}} 

\newcommand{\black}{\black}

\newcommand{\intint}[1]{\llbracket 1,#1 \rrbracket}
\newcommand{\intinta}[2]{\llbracket #1,#2 \rrbracket}

\newcommand{\yplus}{Y^+}
\newcommand{\yminus}{Y^-}

\newcommand{\yyplus}{\mathcal{Y}^{+}}
\newcommand{\yyminus}{\mathcal{Y}^{-}}
\newcommand{\yypm}{\mathcal{Y}^{\pm}}

\title{Directed Spatial Permutations on Asymmetric Tori}

\author{Alan Hammond}
\address{UC Berkeley}

\author{Tyler Helmuth}
\address{Durham University}

\date{17 June 2024}

\begin{document}

\maketitle

\begin{abstract}
  We investigate a model of random spatial permutations on
  two-dimensional tori, and establish that the joint distribution of
  large cycles is asymptotically given by the Poisson--Dirichlet
  distribution with parameter one. The asymmetry of the tori we
  consider leads to a spatial bias in the permutations, and this
  allows for a simple argument to deduce the existence of mesoscopic
  cycles. The main challenge is to leverage this mesoscopic structure
  to establish the existence and distribution of macroscopic
  cycles. We achieve this by a dynamical resampling argument in
  conjunction with a method developed by Schramm for the study of
  random transpositions on the complete graph. Our dynamical analysis
  implements generic heuristics for the occurrence of the
  Poisson--Dirichlet distribution in random spatial permutations, and
  hence may be of more general interest.
\end{abstract}

\section{Introduction}
\label{sec:todo}

\emph{Random permutations} arise in many contexts, with specific
applications leading to different choices of laws. It was observed
long ago that random \emph{spatial} permutations are relevant for the
study of low-temperature condensed matter physics, e.g., the
superfluid transition of helium-4~\cite{Feynman}. The
adjective spatial indicates that the law is biased by the geometry of
the system under study. In physical contexts the spatial bias tends to
zero as the temperature tends to zero.

To make this more concrete, consider the discrete setting of a finite
graph $G=(V,E)$, and let $\beta\geq 0$ represent inverse
temperature. One studies a law $\mu_{\beta}$ on permutations of $V$
that is the uniform distribution if $\beta=\infty$, but otherwise is
biased by the geometry of $G$. For example, adjacent vertices may be
more likely to be in a common cycle than distant vertices. It is
natural to wonder if $\mu_{\beta}$ with $\beta\gg 1$ retains features
of the uniform measure on permutations of $V$. Of particular interest
is the joint distribution of the largest cycles. The question is then
whether or not the largest cycles follow $\mathsf{PD}(1)$, the
Poisson--Dirichlet distribution with parameter $1$.  A precise
definition of $\mathsf{PD}(1)$ will be recalled below, along with the
fact that it describes the large cycles of a uniformly random
permutation. For a broader treatment of the Poisson--Dirichlet
distributions and generalizations thereof, see~\cite{feng}. 

Physically motivated questions about the existence and distribution of
large cycles began to be investigated mathematically roughly 30 years
ago~\cite{Toth1993,AN,Ueltschi}. For a survey from the perspective of
quantum spin systems see~\cite{GUW}. This interest can be traced to
T\'{o}th~\cite{Toth1993}, who formulated an appealing conjecture about
the existence of infinite cycles in the random transposition
model. This question has inspired a great deal of work, especially on
trees~\cite{Angel2003,Hammond2015,HamHeg,VJB} and in the mean-field
(complete graph) setting~\cite{Schramm2011,Berestycki,BK}. There has
been remarkable recent progress on T\'{o}th's conjecture by Elboim and
Sly~\cite{ElboimSly}. Macroscopic (i.e., positive density) cycles have
also recently been proven to exist on the Hamming cube~\cite{AKM} and
random regular graphs~\cite{Poudevigne}, or when reflection positivity
is available~\cite{QT,QT2}. Further related results
include~\cite{KMU,IoffeToth}. Proving the existence of macroscopic
cycles in models that are well-away from mean field, or when
reflection positivity is not available, remains a challenging problem.

A major conceptual advance concerning the distribution of large cycles
was made by Schramm~\cite{Schramm2011} in the mean-field
setting. Further progress in studying the joint distribution of large
cycles has since been made when the model under consideration possesses
integrable features~\cite{EP,BU} or is mean
field~\cite{BKLM}. Schramm's dynamical analysis suggests, however, the
presence of $\mathsf{PD}(1)$ distributed large cycles in many models
of random spatial permutations, whether or not integrable features are
present. Heuristic arguments along these lines have been presented
in~\cite[Section~8]{GUW}. Our goal is to rigorously implement such a
heuristic for a model of random spatial permutations whose relative
simplicity permits an analysis that is not too cluttered by
technicalities. The dynamical technique implemented here may help in
uncovering $\mathsf{PD}(1)$ distributed cycles in other models within the
broad class of random spatial permutations in which these statistics
are expected, and we hope to attract attention to these problems. We
indicate some directions for future study in \Cref{sec:further}.

\subsection{Directed Spatial Permutations: Model and Results}
\label{sec:model}

Let $C_{n}$ denote the cyclic graph on $n$ vertices. Let $\T_{n,m}$
denote the Cartesian product of $C_{n}$ and $C_{m}$, i.e.,
$\T_{n,m} \cong \Z^{2} / (n\Z \times m\Z)$ is an $n$ by $m$ subgraph
of $\Z^{2}$ with periodic boundary conditions (a torus
graph). Vertices in $\T_{n,m}$ will be denoted
$\V{x} = (x_{1},x_{2})$. Expressions in coordinates will always be
interpreted modulo the dimensions of the torus:
\begin{equation*}
\V{x}=(x_{1},x_{2}) = (x_{1}\!\!\!\mod n, x_{2}\!\!\!\mod m).
\end{equation*}

A bijection $\Phi \colon \T_{n,m} \to \T_{n,m}$ determines a
discrete-time dynamical system. The orbits under $\Phi$ are cyclic
subgraphs of $\T_{n,m}$ which will be called \emph{cycles}. We write
$|\cycle|$ for the number of vertices (equivalently, edges) in a cycle
$\cycle$. The aim of this article is to study the joint statistics of
the large cycles of a family of random bijections defined as follows.

Let $\phi$ be the single-step distribution of a one-dimensional 
simple symmetric random walk, i.e., $\phi$ takes values in
$\{-1,0,1\}$ and $\phi(1)=\phi(-1)=a$. Let
$(\tilde \phi_{\V{x}})_{\V{x}\in \T_{n,m}}$ be
a collection of IID random variables, each distributed as $\phi$. 
Define $\tilde\Phi \colon \T_{n,m}\to \T_{n,m}$ by
\begin{equation}
  \label{eq:Bij-Def}
  \tilde{\Phi}( \V{x} ) = (x_{1}+1,x_{2}+\tilde \phi_{\V{x}}+1). 
\end{equation}
When $n\neq m$ the presence of $+1$ in the second coordinate
introduces a bias that will play an important role in what follows.

Note that $\tilde{\Phi}$ given by~\eqref{eq:Bij-Def} will typically
not be a bijection. Let $\Phi$ denote the random variable
$\tilde{\Phi}$ conditioned on being a bijection. We call $\Phi$ a
\emph{directed spatial permutation}. Our symmetry hypothesis on $\phi$
implies that this random variable depends the parameter $a=\phi(1)$ as
well as on $n,m$. We will study $\Phi$ as $m$ tends to infinity under
the following assumptions. The first is mild; the second will be
discussed below.
\begin{enumerate}
\item \textbf{Non-degeneracy}: $0<a<1/2$, and
\item \textbf{Asymmetry of $\T_{n,m}$}: the horizontal dimension $n$
  is given by $n=m+C(m)$ with $C(m) = \ceil{C'\sqrt{m\log m}}$, $C'$
  an absolute constant.
\end{enumerate}

Let $\mathcal{S}_{N}$ denote the set of permutations of $N$
elements. For directed spatial permutations $N=nm$, but the next
definition is more general.
\begin{definition}
  \label{def:Cycle-Structure}
  The \emph{cycle structure} $\cstructure(\pi)$ of a bijection
  $\pi\in \mathcal{S}_N$ is the non-increasing list of normalized
  lengths $(L_{1}, L_{2}, \dots)$ of the cycles of $\pi$, where the
  normalized length $L$ of a cycle $\cycle$ is
  $N^{-1}\abs{\cycle}$. If $\pi$ consists of exactly $k$ distinct
  cycles, then $L_{j}=0$ for $j>k$.
\end{definition}

Before stating our main theorem we recall the definition of the
\emph{Poisson--Dirichlet distribution with parameter $1$}, denoted
$\mathsf{PD}(1)$. $\mathsf{PD}(1)$ is the probability distribution on the
infinite-dimensional simplex
\begin{equation*}
\Delta = \{ (y_{i})_{i\in \nwithoutzero} \mid y_{i}\geq 0,
\sum_{i\in\nwithoutzero}y_{i}=1, y_{1}\geq y_{2}\geq\dots \}
\end{equation*}
defined by the following procedure. Let $(U_{i})_{i\in\nwithoutzero}$
be a sequence of IID random variables, uniform on $\cb{0,1}$. Set
$x_{1}=U_{1}$, and inductively
$x_{i}=U_{i}(1-\sum_{j=1}^{i-1}x_{j})$. Let
$(y_{i})_{i\in\nwithoutzero}$ be the $x_{i}$ in non-increasing
order. Then $\mathsf{PD}(1)$ is the law of
$(y_{i})_{i\in\nwithoutzero}$. It is well-known that uniformly random
permutations have $\mathsf{PD}(1)$ as their limiting cycle
structure~\cite{Kingman,VershikShmidt}.

\begin{theorem}
  \label{thm:PD1}
  Fix $0<a<1/2$, let $C(m)=\ceil{C' \sqrt{m\log m}}$, and suppose $C'$ is a
  large enough absolute constant (independent of $a$). Let $\Phi$ be
  the law of directed spatial permutations on $\T_{n,m}$ with
  $n=m+C(m)$. As $m\to\infty$, the distribution of cycle lengths $\cstructure(\Phi)$
  converges weakly to $\mathsf{PD}(1)$.
\end{theorem}

\subsection{Heuristics and Proof Outline}
\label{sec:sketch}

\subsubsection{Dynamical Heuristic for Poisson--Dirichlet}
\label{sec:heuristic}

The heuristic for why the Poisson--Dirichlet distribution should arise
when studying random spatial permutations is based on a dynamical
perspective on the uniform measure on $\mathcal{S}_{N}$. The relevant
dynamics are the following. 

\begin{definition}
  \label{def:SB-SM}
  The \emph{random transposition model} on $\{1,\dots, N\}$
  is the Markov chain $(\pi_{j})_{j\geq 0}$ with state space
  $\cc S_{N}$ and transition probabilities
  \begin{equation*}
    \P\cb{ \pi_{j+1}=\sigma\pi_{j}} = {N\choose 2}^{-1}, \qquad
    \textrm{$\sigma$ a transposition}.
  \end{equation*}
\end{definition}
By checking the detailed balance equations one can observe that the
stationary distribution of the random transpositions model is the
uniform distribution on $\mathcal{S}_{N}$. Since the limiting cycle
structure of the uniform distribution is $\mathsf{PD}(1)$, the
asymptotic distribution of the cycle structure of the random transposition chain
will essentially be $\mathsf{PD}(1)$.

If we are only concerned with the sizes of cycles, an alternate view
of the random transposition dynamics as a size-biased split-merge
dynamics can be given. The $(j+1)$\textsuperscript{st} step of the
dynamics chooses two (not necessarily distinct) cycles
$\cycle_{1}$ and $\cycle_{2}$ of $\pi_{j}$ with probabilities
$|\cycle_{i}|/N$, and (i) if $\cycle_{1}\neq\cycle_{2}$, the cycles are
merged, and (ii) if instead $\cycle_{1}=\cycle_{2}$, the cycle is
uniformly split into two cycles. We refer to case (i) as a
\emph{merge} and case (ii) as a \emph{split}.

Heuristic conditions for when Poisson--Dirichlet statistics are likely
to occur in models of random spatial permutations can now be
given. Suppose the following conditions hold (possibly in an
approximate sense, which we will not attempt to quantify):
\begin{enumerate}
\item[{[LC]}] Large cycles exist and occupy a positive fraction of space.
\item[{[M]}] Large cycles are well-mixed in space, in the following
  sense. Let $X$ be a random variable that chooses two distinct points
  in space, independently of the cycle configuration. For example, $X$
  could be the endpoints of a uniformly chosen edge. Then we suppose
  that the probability the points of $X$ are contained in two distinct cycles
  $\cycle$, $\cycle'$ is proportional to $|\cycle| |\cycle'|$. 
\item[{[D]}] The law on permutations is invariant under a non-trivial local
  dynamics that updates by (i) choosing two points according to $X$
  and (ii) resampling where the points of $X$ are mapped to with the
  correct marginal law. 
\end{enumerate}
Note that the resampling in {[D]} may result in no change in the
configuration (the resampled configuration could be the same), but if
a change occurs, then the selected cycles merge if they were distinct,
and split if they were the same. By a non-trivial dynamics we
  mean roughly that changes occur with positive probability.

Condition {[LC]} ensures that {[M]} is a non-trivial
statement. Conditions {[M]} and {[D]} indicate that the split-merge
dynamics on large cycles is essentially the same as the random
transposition dynamics, provided the marginal law does not have a
systematic bias towards either splits or merges.  Consequently one
expects $\mathsf{PD}(1)$ to arise under a suitable normalization.  If
a consistent and systematic bias towards splits or merges is present
in the local dynamics of {[D]}, then the parameter $1$ may change. A
more explicit discussion of this heuristic in a particular context can
be found in~\cite[Section~8]{GUW}, see also~\cite[Section~1.2]{BKLM}.

\subsubsection{Proof Outline}
\label{sec:outline}

Our analysis of directed spatial permutations can be viewed as an
implementation of the rough heuristic presented in
\Cref{sec:heuristic}. In fact our implementation is slightly
different, in that we establish versions of the conditions {[LC]},
{[M]} and {[D]} for large (but not macroscopic) pieces of cycles. The
fact that one can argue for $\mathsf{PD}(1)$ via mesoscopic versions
of these conditions should be a rather general fact, and we think
showing how this can be done is one of the main contributions of this
paper.

This section sketches our argument. At several steps in
the full proof we will condition on events of high
probability. We largely omit discussing such technical matters in the following.

\tikzset{
  on each segment/.style={
    decorate,
    decoration={
      show path construction,
      moveto code={},
      lineto code={
        \path [#1]
        (\tikzinputsegmentfirst) -- (\tikzinputsegmentlast);
      },
      curveto code={
        \path [#1] (\tikzinputsegmentfirst)
        .. controls
        (\tikzinputsegmentsupporta) and (\tikzinputsegmentsupportb)
        ..
        (\tikzinputsegmentlast);
      },
      closepath code={
        \path [#1]
        (\tikzinputsegmentfirst) -- (\tikzinputsegmentlast);
      },
    },
  },
  mid arrow/.style={postaction={decorate,decoration={
        markings,
        mark=at position .5 with {\arrow[#1]{Stealth[scale=1.5]}}
      }}},
  mid darrow/.style={postaction={decorate,decoration={
        markings,
        mark=at position .5 with {\arrow[#1]{Stealth[scale=1.5]Stealth[scale=1.5]}}
      }}},
  mid tarrow/.style={postaction={decorate,decoration={
        markings,
        mark=at position .5 with {\arrow[#1]{Stealth[scale=1.5]Stealth[scale=1.5]Stealth[scale=1.5]}}
      }}},
}

\begin{figure}[h]
  \centering
  \begin{tikzpicture}[scale=.8]
    \draw[black, very thick] (0,0) rectangle (13,10);
    \draw[black,|-|] (-.5,0) -- (-.5,10);
    \node[fill=white] at (-.5,5) {$m$};
    \draw[black,|-|] (0,-.5) -- (10,-.5);
    \node[fill=white] at (5,-.5) {$m$};    
    \draw[black,|-|] (10,-.5) -- (13,-.5);
    \node[fill=white] at (11.5,-.5) {$C(m)$};

    \path [draw=black,postaction={on each segment={mid arrow}}]
    (0,0) -- (10,10)
    (10,0) -- (13,3);
    \path [draw=black,postaction={on each segment={mid darrow}}]
    (0,3) -- (7,10)
    (7,0) -- (13,6);
    \path [draw=black,postaction={on each segment={mid tarrow}}]
    (0,6) -- (4,10)
    (4,0) -- (8.5,4.5);
    \path [draw=black,thick,dotted]
    (8.5,4.5) -- (8.75,4.75);
     \path [draw=black,dashed,postaction={on each segment={mid arrow}}]
    (0,1.5) -- (8.5,10)
    (8.5,0) -- (13,4.5);
    \path [draw=black,dashed,postaction={on each segment={mid darrow}}]
    (0,4.5) -- (5.5,10)
    (5.5,0) -- (13,7.5);
    \path [draw=black,dashed,postaction={on each segment={mid tarrow}}]
    (0,7.5) -- (2.5,10)
    (2.5,0) -- (6.9,4.4);
    \path [draw=black,thick,dotted]
    (6.9,4.4) -- (7.15,4.65);
  \end{tikzpicture}
  \caption{Plot of the expected location of portions of two
    cycles (solid and dashed). Horizontal traversals of the system, called strands in the
    main text, share the same number of arrows. After each strand the expected
    vertical displacement is $C(m)$. It takes $m/C(m)$ strands for the
    expected vertical displacement to be zero.}
  \label{fig:displ}
\end{figure}

\textbf{First considerations and dynamics {[D]}.}  In
\Cref{sec:HC-Rep} we examine the typical structure of the random
bijections $\Phi$. The key observation is that the correlation
structure induced by conditioning the random map $\tilde\Phi$ to be a
bijection is fairly simple. First, correlations only exist within each
column of $\T_{n,m}$, as the condition to be a bijection is an
independent condition on each column. By column we mean the set of
vertices with fixed horizontal coordinate (see
\Cref{sec:notation}). Second, the correlations within each column can
be explicitly described on an event of high probability: there is an
explicit description of $\P\cb{ \phi_{\V{x}}=1, \V{x}\in A}$ for a
subset $A$ of vertices in a column. These probabilities give a
complete description of $(\phi_{\V{x}})_{\V{x}\in \T_{n,m}}$ since the
variables form a bijection.  The explicit description is in terms of
the hard-core model from statistical physics.

There are two important consequences of these observations. First, the
correlations in each column are rapidly (exponentially) decaying. This
is a well-known fact about the hard-core model on a one-dimensional
graph.  Second, we can view the underlying randomness as arising from
the hard-core model description. Doing so gives an explicit
dynamics (Glauber dynamics for the hard-core model) that preserves the
law of $\Phi$. We use these dynamics to implement {[D]} from
\Cref{sec:heuristic}.

\textbf{Cycles are not small, and a mesoscopic interpretation of large
  cycles {[LC]}}.  In \Cref{sec:First-Properties} we start to explore
the cycles of $\Phi$ by making the observation that the process of
revealing a cycle initially has the law of a simple random walk. More
precisely, c.f.\ \eqref{eq:Bij-Def}, the first $n$ vertices of a given
cycle have the law of a simple random walk with step distribution
$1+\phi$.  For tori $\T_{n,m}$ with suitable aspect ratios this
implies that cycles cannot be too small. This is the first place where
our assumption on $C(m)$ plays a role.

In more detail, the first $n$ steps are those of a random walk with
drift in the vertical direction. Our choice of $C(m)$ ensures that the
first time a cycle returns to a column it is (with high probability)
at vertical distance of order $C(m)$ from where it started. We may
repeat this argument many times via a union bound, and this leads to a
formalization of the idea that cycles are not too small in terms of
\emph{global traversals}. Roughly speaking, global traversals are
mesoscopic objects that pass through the whole torus
\emph{vertically}, i.e., we follow a cycle until the first time its
vertical displacement in a column is zero. This occurs after roughly
$m/C(m)$ horizontal traversals of the system, see \Cref{fig:displ}. A
precise definition of global traversals is given in
\Cref{sec:First-Properties}; the preceding discussion neglects some
technical aspects of the definition we use in practice. The essential
point is that our assumption on $C(m)$ gives enough \emph{a priori}
control to deduce that each cycle will contain at least one global
traversal. Note that this establishes {[LC]} on a mesoscopic (as
opposed to macroscopic) scale --- the macroscopic scale requires
controlling order $m$ (as opposed to $m/C(m)$) horizontal traversals.

\textbf{Concentration of contacts and the mixing condition {[M]}}.  In
\Cref{sec:Contacts} we leverage the geometry of global traversals to
establish a form of {[M]} at a mesoscopic scale. This is done by
showing that two global traversals encounter one another
(\emph{contact} one another) rather often, and in such a manner that
the number of contacts between two global traversals is a concentrated
random variable. Since the cycles of $\Phi$ are comprised of global
traversals, this means that the cycles of $\Phi$ are well-mixed.

To understand this, consider Figure~\ref{fig:displ}. Given one global
traversal, any other global traversal starts in one of the
``corridors'' determined by the first. Ignoring correlations, the
distance between the two global traversals evolves like a random walk
with increments given by $X+X'$, $X$ and $X'$ distributed as
$\phi$. Due to the corridor structure, this distance between global
traversals should be thought of as being a random walk run on an cycle
of length $\Theta(C(m))$ for approximately $m^{2}/C(m)$ steps. The
number of contacts is the number of times this walk hits zero, and
this is a concentrated random variable. Moreover, it is independent of
the pair of global traversals chosen.

A precise formulation of contacts and concentration is given in
\Cref{sec:Contacts-Defn}. The (exponentially decaying) correlations
that were neglected in the discussion above can be taken into account
without significantly altering the conclusion. The details of this
occupy \Cref{sec:IGC,sec:tbd}; they play no role in the sequel.

\textbf{From {[LC]}, {[M]}, and {[D]} to $\mathsf{PD}(1)$}. In
\Cref{sec:dynamics} we consider the Glauber dynamics which leave
$\Phi$ invariant, and use them to establish \Cref{thm:PD1}. There are
two steps. First, in \Cref{sec:GD-Bijections} we note that under these
dynamics cycles split and merge with one another when $\Phi$ is updated
at a contact between global traversals. Moreover, the dynamics choose
contacts to update uniformly at random.  Since the number
of contacts between any pair of global traversals is concentrated
around the same value, this implies the dynamics are
\emph{effectively} a size-biased split-merge dynamics on the cycles of
$\Phi$.

We know the invariant distribution of random transpositions has
$\mathsf{PD}(1)$ statistics. Our invariant distribution is not that of
the random transposition chain, though, and so we must argue
differently. We do this by utilizing work of Schramm that explained
how size-biased split-merge dynamics lead to $\mathsf{PD}(1)$
\emph{well before} equilibrium is reached in the random transposition
context~\cite{Schramm2011}. We show in \Cref{sec:cycles} that the
errors hidden in the word ``effectively'' in the previous paragraph do
not cause any difficulties in implementing Schramm's argument. The
underlying reason for this is that the distribution of the large
cycles equilibrates extremely rapidly. This key observation, due to
Schramm (see also~\cite{IoffeToth}), indicates that an effective
split-merge dynamics with a low enough error rate leads to the same
distribution as true split-merge.

\subsection{Future Directions}
\label{sec:further}

We have considered directed spatial permutations on asymmetric tori
$\T_{n,m}$ with $n=m+C(m)$ with $C(m)=\ceil{C'\sqrt{m\log m}}$, $C'$
large enough. It is natural to ask what happens for general values of
$C(m)$. While our arguments certainly work for somewhat larger values
of $C(m)$ (i.e., replacing $\sqrt{m\log m}$ with $m^{\frac12+a}$,
$a>0$ not too large), new ideas will be needed at some point. For
example, $C(m)=m$ is essentially the same situation as for $C(m)=0$,
and the \emph{a priori} estimates we use to establish that cycles are
not too small break down. It is with these \emph{a priori} estimates
in mind that we chose the scale $\sqrt{m\log m}$. It seems possible
that our restrictions that the constant $C'$ be large enough could be
weakened at the expense of more technical arguments.

In the setting of a general choice of $C(m)$ it is not entirely clear
what the correct scaling of the cycle structure is nor what the
limiting distribution is. It is at least imaginable that there could
be an analogue of the rational resonances phenomenon found
in~\cite{HammondKenyon}. Note that if $C(m) \gg m^{2}$ then one will
obtain a version of Theorem~\ref{thm:PD1} (the first $n$ steps of a
cycle will be a well-mixed random walk on the $m$-cycle), but if
$n=m+C(m)$ is constant, then one will not. It would be interesting to
develop a more complete understanding of directed spatial permutations
as $C(m)$ varies.

Our results concern the equilibrium distribution of directed spatial
permutations on $\T_{n,m}$, and the proof uses a natural dynamics that
preserve the equilibrium distribution. It seems natural to ask about
the effect of these dynamics started away from equilibrium, in analogy
with Schramm's proof of Aldous's conjecture that $\mathsf{PD}(1)$
emerges in the random transposition model well before the mixing time
occurs~\cite{Schramm2011}. E.g., if one starts from $\phi_{\V{x}}=0$
for all $\V{x}$ and runs the Glauber dynamics underlying our proof, do
$\mathsf{PD}(1)$ statistics arise prior to the mixing time of the
chain?

Our arguments have made use of the geometric structure of $\T_{n,m}$
to give relatively simple arguments for the mixing of large cycles in
the system. Applying our ideas in other settings would be very
interesting. One possibility is to study models of directed spatial
permutations in higher dimensions. Another is to study the properties
of somewhat different models, e.g., random mappings (as opposed to
bijections). Finally, there are well-known models like the random
stirring model or loop representations of quantum spin systems,
see~\cite{GUW}.

\subsection{Conventions and Notation}
\label{sec:notation}

We write $\nwithoutzero=\{1,2,3,\dots\}$ and
$\nwithzero=\{0,1,2,\dots\}$. Discrete intervals are denoted
$\ccb{a,b} = \{a,a+1,\dots, b-1,b\}$ for $a<b\in \nwithzero$, and we
abbreviate $\ccb{b} = \ccb{0,b}$. For typographic convenience we
sometimes write $a\wedge b$ for $\min\{a,b\}$. We use standard
asymptotic notation: $f(x)=O(g(x))$ if there exists a $C>0$ such that
$|f(x)|\leq C g(x)$, and $f(x) = \Theta(g(x))$ if $f(x)=O(g(x))$ and
$g(x)=O(f(x))$.

Let $\pi_{i}$ denote the projection from $\T_{n,m}$ to the
$i$\textsuperscript{th} coordinate, i.e., $\pi_{i}(\V{x})=x_{i}$ for
$i=1,2$.  The \emph{$j$\textsuperscript{th} column} of $\T_{n,m}$ is
$\pi_{1}^{-1}(j)$ for $j\in \ccb{n}$; the
\emph{$j$\textsuperscript{th} row} is $\pi_{2}^{-1}(j)$,
$0\leq j\leq m-1$.

A \emph{geometric random variable $W$ with success probability $p$}
satisfies $\P\cb{W=n}=(1-p)^{n-1}p$, i.e., $W$ is the trial on which
success occurs. A random variable $X$ is \emph{sub-exponential} if
there is a $c>0$ such that $\P\cb{|X|\geq t}\leq 2\exp\{-ct\}$ for all
$t\geq 0$.

In many places our arguments will require that the parameter $m$
controlling the size of $\T_{n,m}$ is large enough. We do not write
this explicitly in our hypotheses.

\section{The typical law of $\Phi$}
\label{sec:HC-Rep}

The collection of IID random variables
$(\tilde \phi_{\V{x}})_{\V{x}\in \T_{n,m}}$ define a map
$\tilde{\Phi}$ by
$\tilde{\Phi}(\V{x})=(x_{1}+1, x_{2}+\tilde \phi_{\V{x}}+1)$. Recall
that $\tilde\phi$ are symmetric on $\{-1,0,1\}$, that we assume
$\phi(1)=a\in\ob{0,\frac12}$, and that $\Phi$ denotes $\tilde\Phi$
conditioned to be a bijection. We similarly write
$(\phi_{\V{x}})_{\V{x}\in \T_{n,m}}$ to denote
$(\tilde\phi_{\V{x}})_{\V{x}\in \T_{n,m}}$ conditioned on
$\tilde{\Phi}$ being a bijection.  This conditioning induces
correlations between the $\phi_{\V{x}}$. This section describes these
correlations. More precisely, by further conditioning on an event that
occurs with high probability we will obtain an explicit description of
the resulting correlation structure.

\subsection{No global shifts}
\label{sec:No-shifts}

By a slight abuse of notation, for $j\in\ccb{n}$ we will write
$\tilde \Phi_{j}$ for the $y$-coordinate of the restriction of
$\tilde \Phi$ to the $j$th column, i.e.,
$\tilde \Phi_{j}(k) = k+\phi_{(j,k)}+1$ records the $y$-coordinate of
$\tilde \Phi((j,k))$. We think of $\tilde \Phi_{j}$ as a map from
$C_{m}$ to $C_{m}$.
If $\tilde \Phi$ is a bijection of $\T_{n,m}$, then $\tilde \Phi_{j}$ induces a
bijection of $C_{m}$ for all $j\in\ccb{n}$. Writing $\Phi_{j}$ for the
restriction of $\Phi$ to column $j$, note that $\Phi_{j}$ and
$\Phi_{k}$ are independent for $j\neq k$, as the condition that
$\tilde{\Phi}$ is a bijection is a separate condition for each
column.

\begin{definition}
  The map $\Phi_{j}$ is a \emph{global up shift} if $\phi_{(j,k)}=1$ for all
  $k\in C_{m}$. Column $j$ is a \emph{global down shift} if
  $\phi_{(j,k)}=-1$ for all $k\in C_{m}$.
\end{definition}

Having no global shifts \emph{and} knowing that $\Phi_{j}$ is a
bijection constrains $\Phi_{j}$ severely.

\begin{lemma}
  \label{lem:jump}
  Suppose $\Phi_{j}$ is not a global shift. If
  $\phi_{(j,k)} =  1$, then $\phi_{(j,k+1)} = - 1$. If
  $\phi_{(j,k)} = -1$, then $\phi_{(j,k-1)} = 1$.
\end{lemma}
\begin{proof}
  The proofs being similar, we derive only the first statement.
  Suppose that $\phi_{(j,k)} =1$ for some $k\in \ccb{n}$. We can make
  the following inferences from $\Phi_{j}$ being a bijection.  The
  value $\phi_{(j,k+1)}$ cannot be zero.  If $\phi_{(j,k+1)}$ equals
  one, then $\phi_{(j,k+2)}$ must also be one.  By induction,
  $\phi_{(j,\ell)}=1$ for all $\ell\in \ccb{n}$, so that $\phi_{j}$ is
  a global shift, contrary to hypothesis. Thus, $\phi_{(j,k)} = - 1$.
\end{proof}

It will be convenient to think of each $\phi_{(j,k)}$ as representing
an \emph{arrow} that points up (value $1$), down (value $-1$), or
horizontally (value $0$). We say the arrows $\phi_{(j,k)}$ and
$\phi_{(j,k+1)}$ \emph{swap} if $\phi_{(j,k)}$ points up and
$\phi_{(j,k+1)}$ points down. They are \emph{parallel} if they are
both horizontal.

\begin{lemma}
  \label{lem:No-Shift}
  Under $\Phi$, the probability of a global shift occurring is at most
  $ne^{-cm}$ for some $c=c(a)>0$.
\end{lemma}
\begin{proof}
  Assume $m$ is even; a similar argument applies to $n$ odd. The
  probability that the unconditioned map $\tilde\Phi$ has a global
  shift in column $j$ is $2a^{m} = 2(a^{2})^{m/2}$, where we recall
  $a=\phi(1)\in \ob{0,\frac12}$.  Consider the two partitions of the
  vertices in $\pi_{1}^{-1}(j)$ into adjacent pairs
  $\{(j,k),(j,k+1)\}$. For each choice of if these pairs of arrows
  swap or are parallel the resulting map is a bijection. Hence the
  probability of a bijection under $\tilde\Phi$ is at least
  $2((1-2a)^{2}+a^{2})^{m/2}$. This proves that conditionally on being
  a bijection, the probability of a global shift in column $j$ is
  exponentially decaying in $m$ if $0<a<1/2$. The lemma follows by a
  union bound over the $n$ columns of $\T_{n,m}$.
\end{proof}

\subsection{The underlying hard-core model}
\label{sec:HC}

We recall the definition of the hard-core model, which we will
shortly see is connected with the bijection $\Phi$. An
\emph{independent set $I\subset V$} of a graph $G=(V,E)$ is a subset
of vertices such that $u,v\in I$ implies that the edge $\{u,v\}$ is
not in $E$. Write $\cc I$ for the set of independent sets of $G$. We
sometimes identify elements of $\cc I$ with vectors
$\sigma = (\sigma_{x})_{x\in V}\in \{0,1\}^{V}$, with $\sigma_{x}=1$
representing that $x$ is in the independent set.
\begin{definition}
  The \emph{hard-core model} with \emph{activity} $\lambda>0$ on a graph
  $G = (V,E)$ is the probability measure $\P$ on $2^{V}$ defined by
  \begin{equation}
    \label{eq:HCM}
    \P\cb{A} \propto \indicthat{A\in \cc I} \lambda^{\abs{A}},
  \end{equation}
  where $\abs{A}$ denotes the cardinality of $A\subset V$.
\end{definition}

\begin{figure}[h]
  \centering
  \begin{tikzpicture}[dualnode/.style={circle,draw,fill=black!80,
      inner sep = 1.4pt}, scale=1.5]
    \draw[step=1,gray,dashed] (0,1) grid (6,5);
    \foreach \x in {0,...,6}{
      \foreach \y in {1,...,5}{
        \node[circle,draw,fill=gray!20,color=gray,inner sep = 1.2pt] () at (\x,\y) {};}}
    \foreach \x in {0,...,6}{
      \foreach \y in {1,...,4}{
        \node [dualnode]  (\x\y) at (\x,\y+.5) {};}
      \foreach \y in {1,...,3}{
        \draw[black, thick] (\x,\y+.5) to [bend left] (\x,1.5+\y);}}
    \foreach \x in {0,...,5}{
      \foreach \y in {1,...,4}{
        \draw[black,thick] (\x,\y+.5) to (\x+1,\y+.5);}}
    \node[] at (6,2.5) [right] {$(x_{1},x_{2})$};
    \node[] at (6,3) [right] {$(x_{1},x_{2})+\frac{1}{2}$};
    \node[] at (0,3) [left] {$\phantom{(x_{1},x_{2})+\frac{1}{2}}$};
  \end{tikzpicture}
  \caption{A subset of $\T_{n,m}$ is in grey, and a subset of
    $\T_{n,m}^{\star}$ is in black. The graph $\T _{n,m}^{\star}$ has vertices $v_{\V{x}} = ((x_{1},x_{2}),(x_{1},x_{2}+1)) $ given by
    the vertical edges of $\T_{n,m}$. Two vertices $v_{\V{x}}$,
    $v_{\V{x}'}$ form an edge in $\T_{n,m}^{\star}$ if either (i) $x_{1}=x_{1}'$
    and $x_{2}-x_{2}'\in \{\pm 1\}$ or (ii) $x_{2}=x_{2}'$ and $x_{1}-x_{1}'\in\{\pm 1\}$.}
  \label{fig:Dual}
\end{figure}

Let $\T^{\star}_{n,m} \cong C_{n}\times C_{m}^{\star}$ be the dual of
$\T_{n,m}$ (we write $C_{m}^{\star}\cong C_{m}$ for the dual of
$C_{m}$).  See \Cref{fig:Dual} and the accompanying caption for a
depiction and precise description.  It is convenient to identify the
vertices of $\T^{\star}_{n,m}$ with the midpoints of the vertical
edges of $\T_{n,m}$, and $\V{x}+\frac{1}{2}$ will denote the vertex
$((x_{1},x_{2}),(x_{1},x_{2}+1))\in \T_{n,m}^{\star}$ for
$\V{x}\in \T_{n,m}$. See \Cref{fig:Dual}.
Recall the definition of swapping arrows from
\Cref{sec:No-shifts}. It is convenient to extend this terminology to
the vertices of $\T^{\star}_{n,m}$.
\begin{definition}
  A \emph{swap} occurs on the edge $\V{x} + \frac{1}{2}$ if
  $\big( \phi_{x_{1}}(x_{2}) , \phi_{x_{1}}(x_{2} +1) \big)$ equals $(1,-1)$.
\end{definition}

\begin{figure}[h]
  \centering
  \begin{tikzpicture}
          \draw[gray,dashed] (0,0) -- (0,5);
      \draw[gray,dashed] (1,0) -- (1,5);
      \draw[gray,dashed] (2,0) -- (2,5);
      \draw[gray,dashed] (3,0) -- (3,5);
      
      \draw[gray,dashed] (0,0) -- (3,0);
      \draw[gray,dashed] (0,1) -- (3,1);
      \draw[gray,dashed] (0,2) -- (3,2);
      \draw[gray,dashed] (0,3) -- (3,3);
      \draw[gray,dashed] (0,4) -- (3,4);
      \draw[gray,dashed] (0,5) -- (3,5);

    \draw[gray,dotted] (0,-.5) -- (0,0);
    \draw[gray,dotted] (0,5) -- (0,5.5);  
    \draw[gray,dotted] (1,-.5) -- (1,0);
    \draw[gray,dotted] (1,5) -- (1,5.5);
    \draw[gray,dotted] (2,-.5) -- (2,0);
    \draw[gray,dotted] (2,5) -- (2,5.5);
    \draw[gray,dotted] (3,-.5) -- (3,0);
    \draw[gray,dotted] (3,5) -- (3,5.5);

    \draw[gray,dotted] (-.5,0) -- (0,0);
    \draw[gray,dotted] (3,0) -- (3.5,0);
    \draw[gray,dotted] (-.5,1) -- (0,1);
    \draw[gray,dotted] (3,1) -- (3.5,1);
    \draw[gray,dotted] (-.5,2) -- (0,2);
    \draw[gray,dotted] (3,2) -- (3.5,2);
    \draw[gray,dotted] (-.5,3) -- (0,3);
    \draw[gray,dotted] (3,3) -- (3.5,3);
    \draw[gray,dotted] (-.5,4) -- (0,4);
    \draw[gray,dotted] (3,4) -- (3.5,4);
    \draw[gray,dotted] (-.5,5) -- (0,5);
    \draw[gray,dotted] (3,5) -- (3.5,5);
    
    \node[circle,draw,fill=gray!20,color=gray,inner sep = 1.2pt] () at
    (0,0) {};
    \node[circle,draw,fill=gray!20,color=gray,inner sep = 1.2pt]
    () at (0,1) {};
    \node[circle,draw,fill=gray!20,color=gray,inner sep = 1.2pt] () at
    (0,2) {};
    \node[circle,draw,fill=gray!20,color=gray,inner sep = 1.2pt]
    () at (0,3) {};
    \node[circle,draw,fill=gray!20,color=gray,inner sep = 1.2pt] () at
    (0,4) {};
    \node[circle,draw,fill=gray!20,color=gray,inner sep = 1.2pt]
    () at (0,5) {};
    
    \node[circle,draw,fill=gray!20,color=gray,inner sep = 1.2pt] () at
    (1,0) {};
    \node[circle,draw,fill=gray!20,color=gray,inner sep = 1.2pt]
    () at (1,1) {};
    \node[circle,draw,fill=gray!20,color=gray,inner sep = 1.2pt] () at
    (1,2) {};
    \node[circle,draw,fill=gray!20,color=gray,inner sep = 1.2pt]
    () at (1,3) {};
    \node[circle,draw,fill=gray!20,color=gray,inner sep = 1.2pt] () at
    (1,4) {};
    \node[circle,draw,fill=gray!20,color=gray,inner sep = 1.2pt]
    () at (1,5) {};

    \node[circle,draw,fill=gray!20,color=gray,inner sep = 1.2pt] () at
    (2,0) {};
    \node[circle,draw,fill=gray!20,color=gray,inner sep = 1.2pt]
    () at (2,1) {};
    \node[circle,draw,fill=gray!20,color=gray,inner sep = 1.2pt] () at
    (2,2) {};
    \node[circle,draw,fill=gray!20,color=gray,inner sep = 1.2pt]
    () at (2,3) {};
    \node[circle,draw,fill=gray!20,color=gray,inner sep = 1.2pt] () at
    (2,4) {};
    \node[circle,draw,fill=gray!20,color=gray,inner sep = 1.2pt]
    () at (2,5) {};
    
    \draw[gray,->] (0,0) -- (.9,0);
    \draw[gray,->] (0,1) -- (.9,1);    
    \draw[gray,->] (0,2) -- (.9,2);    
    \draw[gray,->] (0,3) -- (.9,3);    
    \draw[gray,->] (0,4) -- (.9,5);    
    \draw[gray,->] (0,5) -- (.9,4);

    \draw[gray,->] (1,0) -- (1.9,0);
    \draw[gray,->] (1,1) -- (1.9,2);    
    \draw[gray,->] (1,2) -- (1.9,1);    
    \draw[gray,->] (1,3) -- (1.9,3);    
    \draw[gray,->] (1,4) -- (1.9,4);    
    \draw[gray,->] (1,5) -- (1.9,5); 

    \draw[gray,->] (2,0) -- (2.9,0);
    \draw[gray,->] (2,1) -- (2.9,1);    
    \draw[gray,->] (2,2) -- (2.9,2);    
    \draw[gray,->] (2,3) -- (2.9,4);    
    \draw[gray,->] (2,4) -- (2.9,3);    
    \draw[gray,->] (2,5) -- (2.9,5);

    \node[circle,draw,fill=black!20, inner sep = 2pt] () at (0,0.5) {};
    \node[circle,draw,fill=black!20, inner sep = 2pt] () at (0,1.5) {};
    \node[circle,draw,fill=black!20, inner sep = 2pt] () at (0,2.5) {};
    \node[circle,draw,fill=black!20, inner sep = 2pt] () at (0,3.5) {};
    \node[circle,draw,fill=black!80, inner sep = 2pt] () at (0,4.5) {};

    \node[circle,draw,fill=black!20, inner sep = 2pt] () at (1,0.5) {};
    \node[circle,draw,fill=black!80, inner sep = 2pt] () at (1,1.5) {};
    \node[circle,draw,fill=black!20, inner sep = 2pt] () at (1,2.5) {};
    \node[circle,draw,fill=black!20, inner sep = 2pt] () at (1,3.5) {};
    \node[circle,draw,fill=black!20, inner sep = 2pt] () at (1,4.5) {};

    \node[circle,draw,fill=black!20, inner sep = 2pt] () at (2,0.5) {};
    \node[circle,draw,fill=black!20, inner sep = 2pt] () at (2,1.5) {};
    \node[circle,draw,fill=black!20, inner sep = 2pt] () at (2,2.5) {};
    \node[circle,draw,fill=black!80, inner sep = 2pt] () at (2,3.5) {};
    \node[circle,draw,fill=black!20, inner sep = 2pt] () at (2,4.5) {};
  \end{tikzpicture}
  \qquad\qquad\qquad
    \begin{tikzpicture}
      \draw[gray,dashed] (0,0) -- (0,5);
      \draw[gray,dashed] (1,0) -- (1,5);
      \draw[gray,dashed] (2,0) -- (2,5);
      \draw[gray,dashed] (3,0) -- (3,5);
      
      \draw[gray,dashed] (0,0) -- (3,0);
      \draw[gray,dashed] (0,1) -- (3,1);
      \draw[gray,dashed] (0,2) -- (3,2);
      \draw[gray,dashed] (0,3) -- (3,3);
      \draw[gray,dashed] (0,4) -- (3,4);
      \draw[gray,dashed] (0,5) -- (3,5);

    \draw[gray,dotted] (0,-.5) -- (0,0);
    \draw[gray,dotted] (0,5) -- (0,5.5);  
    \draw[gray,dotted] (1,-.5) -- (1,0);
    \draw[gray,dotted] (1,5) -- (1,5.5);
    \draw[gray,dotted] (2,-.5) -- (2,0);
    \draw[gray,dotted] (2,5) -- (2,5.5);
    \draw[gray,dotted] (3,-.5) -- (3,0);
    \draw[gray,dotted] (3,5) -- (3,5.5);

    \draw[gray,dotted] (-.5,0) -- (0,0);
    \draw[gray,dotted] (3,0) -- (3.5,0);
    \draw[gray,dotted] (-.5,1) -- (0,1);
    \draw[gray,dotted] (3,1) -- (3.5,1);
    \draw[gray,dotted] (-.5,2) -- (0,2);
    \draw[gray,dotted] (3,2) -- (3.5,2);
    \draw[gray,dotted] (-.5,3) -- (0,3);
    \draw[gray,dotted] (3,3) -- (3.5,3);
    \draw[gray,dotted] (-.5,4) -- (0,4);
    \draw[gray,dotted] (3,4) -- (3.5,4);
    \draw[gray,dotted] (-.5,5) -- (0,5);
    \draw[gray,dotted] (3,5) -- (3.5,5);

    \node[circle,draw,fill=gray!20,color=gray,inner sep = 1.2pt] () at
    (0,0) {};
    \node[circle,draw,fill=gray!20,color=gray,inner sep = 1.2pt]
    () at (0,1) {};
    \node[circle,draw,fill=gray!20,color=gray,inner sep = 1.2pt] () at
    (0,2) {};
    \node[circle,draw,fill=gray!20,color=gray,inner sep = 1.2pt]
    () at (0,3) {};
    \node[circle,draw,fill=gray!20,color=gray,inner sep = 1.2pt] () at
    (0,4) {};
    \node[circle,draw,fill=gray!20,color=gray,inner sep = 1.2pt]
    () at (0,5) {};
    
    \node[circle,draw,fill=gray!20,color=gray,inner sep = 1.2pt] () at
    (1,0) {};
    \node[circle,draw,fill=gray!20,color=gray,inner sep = 1.2pt]
    () at (1,1) {};
    \node[circle,draw,fill=gray!20,color=gray,inner sep = 1.2pt] () at
    (1,2) {};
    \node[circle,draw,fill=gray!20,color=gray,inner sep = 1.2pt]
    () at (1,3) {};
    \node[circle,draw,fill=gray!20,color=gray,inner sep = 1.2pt] () at
    (1,4) {};
    \node[circle,draw,fill=gray!20,color=gray,inner sep = 1.2pt]
    () at (1,5) {};

    \node[circle,draw,fill=gray!20,color=gray,inner sep = 1.2pt] () at
    (2,0) {};
    \node[circle,draw,fill=gray!20,color=gray,inner sep = 1.2pt]
    () at (2,1) {};
    \node[circle,draw,fill=gray!20,color=gray,inner sep = 1.2pt] () at
    (2,2) {};
    \node[circle,draw,fill=gray!20,color=gray,inner sep = 1.2pt]
    () at (2,3) {};
    \node[circle,draw,fill=gray!20,color=gray,inner sep = 1.2pt] () at
    (2,4) {};
    \node[circle,draw,fill=gray!20,color=gray,inner sep = 1.2pt]
    () at (2,5) {};

    \node[circle,draw,fill=gray!20,color=gray,inner sep = 1.2pt] () at
    (3,0) {};
    \node[circle,draw,fill=gray!20,color=gray,inner sep = 1.2pt]
    () at (3,1) {};
    \node[circle,draw,fill=gray!20,color=gray,inner sep = 1.2pt] () at
    (3,2) {};
    \node[circle,draw,fill=gray!20,color=gray,inner sep = 1.2pt]
    () at (3,3) {};
    \node[circle,draw,fill=gray!20,color=gray,inner sep = 1.2pt] () at
    (3,4) {};
    \node[circle,draw,fill=gray!20,color=gray,inner sep = 1.2pt]
    () at (3,5) {};
    
    \draw[gray,->] (0,0) -- (.9,0);
    \draw[gray,->] (0,1) -- (.9,1);    
    \draw[gray,->] (0,2) -- (.9,2);    
    \draw[gray,->] (0,3) -- (.9,3);    
    \draw[gray,->] (0,4) -- (.9,5);    
    \draw[gray,->] (0,5) -- (.9,4);

    \draw[gray,->] (1,0) -- (1.9,0);
    \draw[gray,->] (1,1) -- (1.9,1);    
    \draw[gray,->] (1,2) -- (1.9,2);    
    \draw[gray,->] (1,3) -- (1.9,3);    
    \draw[gray,->] (1,4) -- (1.9,4);    
    \draw[gray,->] (1,5) -- (1.9,5); 

    \draw[gray,->] (2,0) -- (2.9,0);
    \draw[gray,->] (2,1) -- (2.9,1);    
    \draw[gray,->] (2,2) -- (2.9,2);    
    \draw[gray,->] (2,3) -- (2.9,4);    
    \draw[gray,->] (2,4) -- (2.9,3);    
    \draw[gray,->] (2,5) -- (2.9,5);

    \node[circle,draw,fill=black!20, inner sep = 2pt] () at (0,0.5) {};
    \node[circle,draw,fill=black!20, inner sep = 2pt] () at (0,1.5) {};
    \node[circle,draw,fill=black!20, inner sep = 2pt] () at (0,2.5) {};
    \node[circle,draw,fill=black!20, inner sep = 2pt] () at (0,3.5) {};
    \node[circle,draw,fill=black!80, inner sep = 2pt] () at (0,4.5) {};

    \node[circle,draw,fill=black!20, inner sep = 2pt] () at (1,0.5) {};
    \node[circle,draw,fill=black!20, inner sep = 2pt] () at (1,1.5) {};
    \node[circle,draw,fill=black!20, inner sep = 2pt] () at (1,2.5) {};
    \node[circle,draw,fill=black!20, inner sep = 2pt] () at (1,3.5) {};
    \node[circle,draw,fill=black!20, inner sep = 2pt] () at (1,4.5) {};

    \node[circle,draw,fill=black!20, inner sep = 2pt] () at (2,0.5) {};
    \node[circle,draw,fill=black!20, inner sep = 2pt] () at (2,1.5) {};
    \node[circle,draw,fill=black!20, inner sep = 2pt] () at (2,2.5) {};
    \node[circle,draw,fill=black!80, inner sep = 2pt] () at (2,3.5) {};
    \node[circle,draw,fill=black!20, inner sep = 2pt] () at (2,4.5) {};
  \end{tikzpicture}
  \caption{Two illustrations of the relation between the hard-core
    model and the arrows (in grey). The occupied vertices in the hard-core
    model ($\sigma_{x}=1$) are depicted with large black circles, and
    the vacant vertices ($\sigma_{x}=0$) with large shaded
    circles. Small grey dots are vertices of $\T_{n,m}$. Only a subset
    of the hard-core and arrow configurations are depicted. The change in the occupation
    of a vertex between the left- and right-hand illustrations causes
    a split/merge. Note that the deterministic shifts by $+1$ in the
    $y$-coordinates in the definition of $\Phi$ are not illustrated.}
  \label{fig:HC-arrows}
\end{figure}

Given a bijection of $\T_{n,m}$, consider the subset of vertices of
$\T^{\star}_{n,m}$ at which a swap occurs on the corresponding edge of
$\T_{n,m}$.  This defines a map between bijections $\Phi$ of
$\T_{n,m}$ and subsets of $V(T^{\star}_{n,m})$.

\begin{proposition}
  \label{prop:HC-Rep}
  Consider the set of bijections $\Phi$ that do not contain a global
  shift. The induced law on subsets of $V(\T^{\star}_{n,m})$ is the
  product of independent hard-core model measures with activity
  $\lambda = \lambda(a)=a^{2}/(1-a)^{2}$ on each
  column $C_{m}^{\star}$ of $\T^{\star}_{n,m}$.
\end{proposition}
\begin{proof}
  By \Cref{lem:jump}, every vertex $\V{x}$ of $\T_{n,m}$ not assigned
  a horizontal arrow (i.e., $\phi_{\V{x}}\neq 0$) is part of exactly
  one swap.  Thus every configuration which is given non-zero
  probability is an independent set in each column $C_{m}^{\star}$ of
  $\T^{\star}_{n,m}$. Since any independent set corresponds to a
  possible set of swaps, every independent set occurs with positive
  probability.

  Because $\Phi_{j}$ and $\Phi_{k}$ are independent for
  $j\neq k$, it is enough to show that the law on a single
  column $C_{m}^{\star}$ is given by (\ref{eq:HCM}). To this end, note that,
  for $A\subset V(C_{m}^{\star})$ an independent set,
  \begin{equation*}
    \P \cb{A} \propto a^{2\abs{A}}(1-a)^{n-2\abs{A}} 
    \propto \ob{ \frac{a^{2}}{(1-a)^{2}}}^{\abs{A}} \, . \qedhere
  \end{equation*}
\end{proof}

For the rest of this article we will denote the activity of the
hard-core model induced by our model of random bijections by 
\begin{equation}
  \label{eq:HCactivity}
  \lambda = \lambda(a) = a^{2}(1-a)^{-2}\,.
\end{equation}

\begin{remark}
  \label{rem:HC-Perm}
  The proof of \Cref{prop:HC-Rep} used a bijection between configurations
  of $\Phi$ (conditioned on having no global shifts) and independent
  sets. Identifying independent sets with vectors
  $\sigma = (\sigma_{x})_{x\in V}$, this bijection had the property
  that $\phi_{\V{x}}=1$ iff $\sigma_{\V{x}+\frac12}=1$,
  $\phi_{\V{x}}=-1$ iff $\sigma_{\V{x}-\frac12}=1$, and
  $\phi_{\V{x}}=0$ iff
  $\sigma_{\V{x}+\frac12}=\sigma_{\V{x}-\frac12}=0$. In other words,
  there is a perfect coupling between the hard-core model and $\Phi$
  conditioned on having no global shifts.
\end{remark}

\begin{remark}
  \label{rem:NGS}
  \Cref{lem:No-Shift} implies that conditioning on the non-occurrence
  of a global shift is asymptotically irrelevant if $n=n(m)$ is
  sub-exponential in $m$. Under this conditioning the coupling of \Cref{rem:HC-Perm}
  means that we can use the corresponding hard-core models as the
  underlying source of randomness of $\Phi$. We will do this
  frequently in what follows.
\end{remark}

\subsection{Correlation decay in the hard-core model}
\label{sec:HC-CD}

Since $C_{m}$ is one-dimensional, the hard-core model on $C_{m}$ has
rapidly decaying correlations. For later use, the next lemma
formalizes this.  For $A,B\subset C_{m}$ let
$\dist(A,B) = \min_{a\in A, b\in B} \abs{a-b}$, where $\abs{a-b}$ is
the graph distance between $a$ and $b$. Given a vector
$\sigma = (\sigma_{x})_{x\in V}$ 
and $A\subset V$ let $\sigma_{A} = (\sigma_{x})_{x\in A}$.

\begin{lemma}
  \label{lem:HC-CD}
  Let $\P$ be the law of the hard-core model on $C_{m}$, and let
  $\sigma$ denote the corresponding random independent set.  There are
  positive constants $c_{1},c_{2}$, depending on $\lambda$ but not on
  $m$ such that
  \begin{equation}
    \label{eq:HC-CD}
    \Big\vert \, \P\cb{\sigma_{A}=\tilde{\sigma}_{A}, \sigma_{B}=\tilde{\sigma}_{B}} -
      \P\cb{\sigma_{A}=\tilde{\sigma}_{A}}\P\cb{\sigma_{B} 
        = \tilde{\sigma}_{B}} \, \Big\vert \, \leq \, c_{1}e^{-c_{2}\dist (A,B)} \, .
  \end{equation}
  for all subsets $A,B\subset C_{m}$ such that $A$ is connected and all configurations
  $\tilde{\sigma}$. 
\end{lemma}
\begin{proof}
  This is a well-known fact, and we omit a detailed proof. It can be
  derived by using the observation that the partition function
  (normalizing constant) $Z_{n}(\lambda)$ for the hard-core measure on
  a linear graph with $n$ vertices satisfies the recurrence
  $Z_{n}(\lambda)=Z_{n-1}(\lambda)+\lambda Z_{n-2}(\lambda)$, which has an
  elementary solution.
\end{proof}

The next lemma records how correlation decay in the hard-core model
translates into correlation decay for the random variables
$\phi_{\V{x}}$ that comprise $\Phi$.

\begin{lemma}
  \label{lem:RW-CD}
  Condition on $\Phi$ having no global shifts. There exist positive
  constants $c_{1},c_{2}$ depending on $\lambda$ but not on $m$ such
  that
    \begin{equation}
    \label{eq:RW-CD}
    \abs{ \P\cb{\phi_{A}=\tilde{\phi}_{A}, \phi_{B}=\tilde{\phi}_{B}} -
    \P\cb{\phi_{A}=\tilde{\phi}_{A}}  \P\cb{\phi_{B}=\tilde{\phi}_{B}}} 
        \leq c_{1}e^{-c_{2}\dist (A,B)},
  \end{equation}
  for all subsets $A,B$ belonging to single columns such that $A$ is
  connected. 
\end{lemma}
\begin{proof}
  The claim is immediate by independence if $A,B$ are in disjoint
  columns. Otherwise, let $\P_{HC}$ be the law of independent hard-core models
  on each column of $\T_{n,m}^{\star}$, and let $\P$ be the induced
  law on $\{\phi_{\V{x}}\}_{\V{x}\in \T_{n,m}}$. Then by
  \Cref{rem:HC-Perm}, and letting $\sigma_{A}$ denote the hard-core
  configuration corresponding to the arrows specified by $\tilde \phi_A$,
  \begin{equation}
    \label{eq:TD}
    \P \big[ \, \phi_{\V{x}} = w \, \big\vert \, \sigma_{A} \, \big] \, = \, 
    \begin{cases}
     \, \P_{HC} \big[ \, \sigma_{\V{x}-\frac{1}{2}}=1 \, \big\vert \, \sigma_{A} \, \big]  & w=-1 \\
    \,  \P_{HC} \big[ \, \sigma_{\V{x}+\frac{1}{2}}=1  \, \big\vert  \,  \sigma_{A} \, \big] & w=1 \\
    \,  \P_{HC} \big[ \, \sigma_{\V{x}-\frac{1}{2}}=\sigma_{\V{x}+\frac{1}{2}}=0
        \, \big\vert \,   \sigma_{A} \, \big] & w = 0 
    \end{cases},
  \end{equation}
  and similarly for computations of $\P
  \cb{\phi_{\V{x}}=w,\phi_{\V{y}}=w'|\sigma_A}$. If
  $A,B$ contain vertices in the same column the claim now follows from
  \Cref{lem:HC-CD}: the probability of a configuration on the
  interval $A$ is only affected by the configuration of
  $B$ closest to each side of the interval.
\end{proof}

In \Cref{sec:IGC} we will make use of the hard-core model on $\Z$. Let
$\P_{m}$ denote the hard-core measure on $C_{m}$ with a fixed activity
$\lambda$. Correlation decay allows $\P_{\Z}$, the hard-core model on
$\mathbb{Z}$, to be defined as the unique weak limit of the $\P_{m}$
as $m\to\infty$. $\P_{\Z}$ is translation invariant and
inherits the correlation decay properties of the finite-volume
measures. In particular, $\P_{\Z}\cb{\sigma_{x}=1}$ exists, is positive if
$\lambda>0$, and is independent of $x$.

\subsection{Glauber dynamics for the hard-core model}
\label{sec:HC-GD}

Let $G=(V,E)$ be a finite graph. The \emph{Glauber dynamics} for the
hard-core model are a Markov chain on $\{0,1\}^{V}$ whose equilibrium
distribution is that of the hard-core model on $G$. Let
$\sigma = (\sigma_{v})_{v\in V}$ denote an independent set. A single
step of the chain is defined as follows.
\begin{enumerate}
\item Choose a vertex $v\in V$ uniformly at random.
\item Conditionally on $\{\sigma_{x}\mid \{x,v\}\in E\}$, resample
  $\sigma_{v}$ according to the hard-core measure.
\end{enumerate}
The claim that the hard-core measure is the equilibrium distribution
of the chain follows by checking the detailed balance equations. 

\begin{remark}
  \label{rem:Glauber-SM}
  Specialize to the case $G=C_{m}^{\star}$, and recall the bijection
  between hard-core and arrow configurations from \Cref{rem:HC-Perm}.
  A Glauber dynamics update at $v$ can change the value of
  $\sigma_{v}$ only if $\sigma_{v+1}=\sigma_{v-1}=0$. If
  $\sigma_{v}=0$ as well, then the arrows at $v\,\pm \frac{1}{2}$ are
  parallel, while if $\sigma_{v}=1$ there is a swap at $v$. If a
  Glauber update at $v$ results in a change of $\sigma_{v}$ then
  either parallel arrows become swapping arrows or vice versa.
  Viewing these arrows as being single steps in the cycle(s) of
  $v\,\pm\frac12$, it follows that this update merges the cycles if
  they are distinct, and splits the cycle otherwise. See
    Figure~\ref{fig:HC-arrows}. This induced
  dynamics on cycles will be our main tool for studying the cycle
  structure of $\Phi$.
\end{remark}

\section{Equilibrium properties I: strands and global traversals}
\label{sec:First-Properties}

The main aim of this section is to define convenient geometric
structures that allow us reduce statements about cycles to
statements about one-dimensional random walks. The key outputs of this
section will be the idea of a \emph{global traversal}, that all cycles
are comprised of at least one global traversal, and that the number of
global traversals in a cycle is an accurate proxy for the size of a
cycle.

Henceforth we set $C(m)=\ceil{C'\sqrt{m\log m}}$ for a constant $C'$
that will be determined later (see \Cref{cor:All-Separated} and
\Cref{rem:Sdredef}).

\subsection{Random walk interpretation: strands}
\label{sec:RW}

The cycles of a random bijection $\Phi$ can be thought of as closed
random walk paths.  The lack of independence between the random
variables $(\phi_{\V{x}})_{\V{x}\in\T_{n,m}}$ that comprise $\Phi$
prevents direct control of these closed paths. This section introduces
\emph{strands} as a probabilistically convenient tool that will be
used to study the local geometric structure of cycles.

\begin{definition}
  Fix $\V{x}\in \T_{n,m}$, and let $\V{x}_{k} = \Phi^{k}(\V{x})$ for
  $k=0, \dots, n-1$. The \emph{strand $\st{\V{x}}$ starting at
    $\V{x}$} is the sequence $(\V{x}_k)_{k=0}^{n-1}$. Explicitly, 
\begin{equation}
  \label{eq:strand-explicit}
   \st{\V{x}}=\ob{ \ob{x_{1}+j,x_{2}+\sum_{k=0}^{j-1}(\phi_{\V{x}_{k}} + 1)}}_{j=0}^{n-1}.
\end{equation}
\end{definition}

The probabilistic content of a strand is contained in the increments
$\phi_{\V{x}_{k}}$, which encode a one-dimensional random walk
trajectory $\sti{\V{x}}$,
\begin{equation}
  \label{eq:increment-explicit}
  \sti{\V{x}} = \ob{x_{2}+\sum_{k=0}^{j-1}\phi_{\V{x}_{k}}}_{j=0}^{n-1},
\end{equation}
with the convention that the sum from $0$ to $-1$ is zero.  The
independence of the increments in distinct columns along a strand
gives a simple but important observation.
\begin{lemma}
  \label{lem:increment}
  Condition on the event that $\Phi$ contains no global shift. 
  $\sti{\V{x}}$ is equal in law to an $n$-step random walk with IID
  increments in $\{-1,0,1\}$. The law of the increments does not
  depend on $\V{x}$, is symmetric, and has finite and non-zero
  variance.
\end{lemma}
\begin{proof}
  Since the conditioning does not affect the
  distribution of a single arrow in a column, the conclusion follows.
  By Proposition~\ref{prop:HC-Rep} and
  Remark~\ref{rem:HC-Perm} the probability $\P\cb{\phi_{\V{x}}=1}$
  that an increment at $\V{x}$ is $+1$ is the probability that
  $\V{x}-\frac12$ is occupied in the hard-core model on $C_{m}$, which
  is independent of $\V{x}$. By translation invariance of the
  hard-core model this is equal to $\P\cb{\phi_{\V{x}}=-1}$, i.e., the
  probability that $\V{x}+\frac12$ is occupied. Non-degeneracy follows
  as $\P\cb{\phi_{\V{x}}=-1}\in \ob{0,1}$.
\end{proof}

\subsection{Strand geometry}
\label{sec:RW-Objects}

The fact that $\sti{\V{x}}$ is an $n$-step lazy simple random walk
allows for strong \emph{a priori} statements about the geometry of
strands. We record the key lemma in its natural generality.

\begin{lemma}
  \label{lem:Corridor}
  Let $S_{0}=0$, $S_{k} = \sum_{i=1}^{k}X_{i}$, with the $X_{i}$ IID
  and distributed as $X$. Suppose $X$ is symmetric about $0$,
  $\Var{X}=\sigma^{2}>0$, and that $X$ has a finite moment generating
  function.  For $A>0$ and $c>0$ fixed constants the following
  estimate holds:
  \begin{equation}
    \label{eq:Corridor}
    \P \cb{ \max_{0\leq j\leq ct}\abs{S_{j}} \geq A \sqrt{t\log t}}
    \leq (1+o(1)) t^{-\frac{A^{2}}{2c\sigma^{2}}}. 
  \end{equation}
\end{lemma}
\begin{proof}
  Let $a>0$. Note that
  \begin{align*}
         \P\cb{S_{j}\in\cb{-a,a}, 0\leq j\leq k} 
   & = 
    \P\cb{ \max_{0\leq j\leq k} S_{j}\leq a \, , \, 
    \min_{0\leq j\leq k}S_{j}\geq -a} \\
  & =  1 - \P \cb{ \Big\{ \max_{0\leq j\leq k} S_{j} >  a \Big\} \bigcup \Big\{ \min_{0\leq
     j\leq k}S_{j} <  -a \Big\} }
   \\
   & \geq  1- 2\, \P\cb{\max_{0\leq j\leq k}S_{j}>a}
   \\
   & \geq  1- 4\, \P\cb{ S_{k}>a}
  \end{align*}
  where the first inequality is by the union bound and the symmetry of
  the random walk, and the second inequality is by the reflection
  principle $\P\cb{\max_{0\leq j\leq k}S_{j} >a} \leq
  2\P\cb{S_{k}>a}$, see, e.g.,~\cite[Theorem~5.2.7]{Durrett}.

  Upper bounding $\P\cb{S_{k}>a}$ when $a=A\sqrt{t\log t}$ and $k=ct$
  requires a large deviation estimate. We use
  \cite[Theorem~5.23]{Petrov} (and a standard tail estimate for the
  normal distribution) to obtain 
  \begin{equation}
    \label{eq:Petrov}
    \P\cb{S_{ct}>A \sqrt{t\log t}} 
        \leq  (1+o(1)) \exp( - (\frac{A}{\sqrt{c}} \sqrt{\log
      t})^{2}/2\sigma^{2}) 
        =  \big(1+o(1)\big) t^{-\frac{A^{2}}{2c\sigma^{2}}}.\qedhere
  \end{equation}
\end{proof}

Given a vertex $\V{x}$, the cycle $\cycle(\V{x})$ containing $\V{x}$
can be decomposed into a sequence of strands.  The second strand
begins at the endpoint of the first, i.e., at the endpoint of
$\st{\V{x}}$. This inductively decomposes $\cycle(\V{x})$ into
$(\st{\V{x}_{i}})_{i=0}^{k}$ for some $k\in\N$, where $\V{x}_{i}$ is
the initial vertex of the $i$\textsuperscript{th} strand. Each of
these strands are distinct. For $j\in \ccb{0,k}$, the strands
$\st{\V{x}_{j}}$ and $\st{\V{x}_{j+1}}$ will be called
\emph{consecutive}, with $k+1$ interpreted as $0$. The relation of
being consecutive puts a cyclic order on strands that are contained in
the same cycle.

Lemma~\ref{lem:Corridor} implies that consecutive strands are
well-separated if the torus $\T_{n,m}$ is appropriately asymmetric, as
the vertical distance between two consecutive strands in a single
column is equal to the vertical distance between the first and last
points of a single strand. To make a precise statement, we extend the
distance $\dist(A,B)$ between subsets $A,B\subset C_{m}$ to distances
between subsets of $A$ and $B$ of $\T_{n,m}$ as follows. Let $A_{j}$ denote the
subset of $A$ in column $j$, and similarly for $B_{j}$. Define
\begin{equation*}
  \dist^{-}(A,B) = \min_{j\in \ccb{n}} \dist(A_{j},B_{j}), \quad
  \dist^{+}(A,B) = \max_{j\in\ccb{n}} \dist(A_{j},B_{j}).
\end{equation*}

\begin{proposition}
  \label{prop:Strand-Gap}
  Fix $A>0$, and suppose $n=m+C(m)$ with $C'>A$.
  Let $\V{x}_{2}$ be the
  endpoint of the strand $\st{\V{x}_{1}}$.
  Then
  \begin{eqnarray}
    \label{eq:Gap-LB}
    \P \cb{ \dist^{-}(\st{\V{x}_{1}}, \st{\V{x}_{2}}) \leq \floor{C(m)-A\sqrt{m\log m} }}
    &\leq & \big(1+o(1) \big)m^{-\frac{A^{2}-2}{2}}, \\
        \label{eq:Gap-UB}
    \P \cb{ \dist^{+}(\st{\V{x}_{1}}, \st{\V{x}_{2}}) \geq \floor{C(m)+A \sqrt{m\log m}}}
    &\leq &  \big(1+o(1) \big)m^{-\frac{A^{2}-2}{2}},
  \end{eqnarray}
\end{proposition}
\begin{proof}
  We start with \eqref{eq:Gap-LB}.  Let
  $\st{\V{x}_{i}} = \big(\V{x}^{(i)}_{0}, \V{x}^{(i)}_{1}, \dots,
  \V{x}^{(i)}_{n-1} \big)$ for $i=1,2$. Note that
  \begin{equation*}
    \dist^{-}(\st{\V{x}_{1}}, \st{\V{x}_{2}}) \, = \, \min_{i=0,1,\dots, n-1}
    \abs{ \pi_{2}( \V{x}^{(1)}_{i}) - \pi_{2}(\V{x}^{(2)}_{i})} \, ,
  \end{equation*}
  and, (recall~\eqref{eq:increment-explicit})
  \begin{equation*}
    \abs{\pi_{2} (\V{x}^{(1)}_{i}) - \pi_{2}(\V{x}^{(2)}_{i})} \, = \, 
    \abs{\sti{\V{x}_{0}^{(i)}}_{n-1} + C(m)} \mod m. 
  \end{equation*}
  The preceding display and $C(m)=\ceil{C'\sqrt{m\log
      m}}\leq m/3$ implies that the bound 
  \begin{equation*}
  \abs{\pi_{2} (\V{x}^{(1)}_{i}) - \pi_{2}(\V{x}^{(2)}_{i})} \, \leq \,
  \floor{C(m)-A\sqrt{m\log m}}
  \end{equation*}
  can hold only if the endpoint $\sti{\V{x}_{0}^{(i)}}_{n-1}$ of the
  increment random walk $\sti{\V{x}_{0}^{(i)}}$ started at
  $\V{x}_{0}^{i}$ has absolute value at least $A\sqrt{m\log m}$.
  \Cref{lem:increment} shows the hypotheses of \Cref{lem:Corridor} are
  satisfied, and applying \Cref{lem:Corridor} with $c=1$ and $t=n$
  yields
  \begin{equation*}
    \P\cb{ \, \abs{\pi_{2} (\V{x}^{(1)}_{i}) - \pi_{2}(\V{x}^{(2)}_{i})}
      \geq A\sqrt{m\log m} \, }  \,\leq \, \big(1+o(1) \big) m^{-\frac{A^{2}}{2}} \,,
  \end{equation*}
  where we have used that $n=(1+o(1))m$ and that $\sigma^{2}<1$ for
  the strand random walk. Hence by a union bound over the $m$
  consecutive strands in the system,
  \begin{equation*}
    \P \cb{ \, \dist^{-}(\st{\V{x}_{1}}, \st{\V{x}_{2}}) \leq \floor{C(m)-A\sqrt{m\log m}} \, }
   \,  \leq\,  \big(1+o(1) \big) m^{-\frac{A^{2}-2}{2}}.
  \end{equation*}
  The bound (\ref{eq:Gap-UB}) follows, \emph{mutatis mutandis}, from
  the same argument.
\end{proof}

The remainder of this section collects properties of strands that hold
with high probability.

\begin{definition}
  \label{def:S}
  Let $S_{L}(r)$ be the event that all pairs of consecutive strands
  $\st{\V{x}}$, $\st{\tilde{\V{x}}}$ satisfy
  $\dist^{-}(\st{\V{x}},\st{\tilde{\V{x}}})\geq r\sqrt{m\log
  m}$. Formally, if $\V{x}^{+}$ denotes the endpoint of
  $\st{\V{x}}$,
  \begin{equation}
    \label{eq:SL}
    S_{L}(r) = \bigcap_{\V{x}\in \pi_{2}^{-1}(0)} \set{
      \dist^{-}(\st{\V{x}},\st{\V{x}^{+}})\geq r\sqrt{m\log m}}.
  \end{equation}
  Similarly, let $S_{U}(r)$ denote the event that all pairs of
  consecutive strands satisfy $\dist^{+}(\st{\V{x}}$,
  $\st{\tilde{\V{x}}})\leq r\sqrt{m\log m}$.  Let
  $S^{0}(D)=S_{L}(D) \cap S_{U}(4+\frac{3D}{2})$.
\end{definition}

\begin{corollary}
  \label{cor:All-Separated}
  Let $D>0$ and $C'=2+5D/4$. Then
  $\P\cb{ S^{0} (D)} \geq 1- o(m^{-D/2})$.
\end{corollary}
\begin{proof}
  Let $A = 2 + D/4$. The result follows from
  \Cref{prop:Strand-Gap} and the union bound since there are exactly
  $m$ consecutive strands in any configuration.
\end{proof}

\begin{remark}
  \label{rem:C'}
  Henceforth we will assume $C'=C'(D)$ has been chosen as in
  \Cref{cor:All-Separated}.
\end{remark}

Corollary~\ref{cor:All-Separated} implies that we can study $\Phi$
conditionally on $S^{0} (D)$ occurring for some $D>0$, i.e., we can
work on the event that consecutive strands are neither very close nor
very far from one another.

Let $\V{x}'$ be the endpoint of $\st{\V{x}}$. The next lemma says that
if $\V{y}$ is in the same column as $\V{x}$, and $\V{y}$ is not too close to either
$\V{x}$ or $\V{x'}$, then $\st{\V{y}}$ will not get near $\st{\V{x}}$
or $\st{\V{x'}}$.
\begin{lemma}
  \label{lem:safezone}
  Let $\eta'= D/2$. For $D$ large enough there is a $c>0$ such that
  the following holds with probability $1-m^{-cD^{2}}$ on $S^{0}(D)$. Suppose
  $\V{x}$ and $\V{y}$ are in the same column, and
  $d^{-}(\V{x},\V{y})>\eta'\sqrt{m\log m}$. If
  $d^{-}(\V{x}',\V{y})\geq \eta' \sqrt{m\log m}$, then
  $d^{-}(\st{\V{x}'},\st{\V{y}})$ and $d^{-}(\st{\V{x}},\st{\V{y}})$
  are at least $\Theta( D\sqrt{m\log m})$.
\end{lemma}
\begin{proof}
  The event $S^{0} (D)$ entails that
  $d^{-}(\st{\V{x}},\st{\V{x}'})\geq D\sqrt{m\log m}$ for all
  $\V{x}$. By \eqref{eq:Corridor} with $A=\frac{D}{5}$ and a union
  bound, there exists $c > 0$ such that each strand $\sti{\V{x}}$ is
  contained in a corridor of width $\frac{D}{5}\sqrt{m\log m}$ with
  probability $1-m^{-cD^{2}}$ if $D$ is chosen large enough.  Since
  $\eta'= D/2$ the conclusion follows.
\end{proof}
\begin{remark}
  \label{rem:S1}
  In the sequel we set $S^{1}(D)$ to be the event that $S^{0} (D)$
  occurs, that $\Phi$ does not contain a global shift, and that the
  conclusion of Lemma~\ref{lem:safezone} holds. For $D$ large,
  $S^{1}(D)$ occurs with probability $1-o(m^{-\frac{D}{2}})$.
\end{remark}

\subsection{Global traversals and first properties}
\label{sec:GTn}

This section introduces and discusses global traversals, the basic
mesoscopic unit of our arguments.

\subsubsection{Definition of global traversals}
\label{sec:GT-def}

We would like to group strands into consecutive sequences that
`vertically span' the system. The next definition is one way to
achieve this.
\begin{definition}
  \label{def:GTn}
  Let $\Gamma = \Gamma(m) = \floor{\frac{m}{C(m)}-2m^{1/4}}$. The
  \emph{global traversal} $\gt{\V{x}}$ starting at $\V{x}$ is the
  sequence of $\Gamma$ consecutive strands starting with $\st{\V{x}}$,
  provided these strands are all distinct. If they are not distinct we
  say there is no global traversal at $\V{x}$.
\end{definition}
The choice of $\Gamma$ reflects that the expected vertical
displacement along each strand is $C(m)$.  This implies the expected
vertical displacement after $\Gamma+2m^{1/4}$ strands is zero.
Roughly speaking, then, this is the number of strands after which a
cycle has an appreciable chance to close. The omission of $2m^{1/4}$
strands in the definition of a global traversals is for technical
convenience: it ensures that the first and last strands of a global
traversal do not interact with one another.

Given $\V{x}$, recall that $\cycle(\V{x})$ denotes the cycle of $\Phi$
that contains $\V{x}$. Let $K(\V{x})$ denote the integer number of
strands contained in $\cycle(\V{x})$, and set
\begin{align}
  \{\cycle(\V{x})\}_{\gtt}=\frac{K(\V{x})\!\!\!\mod  \Gamma}{\Gamma}\,, \quad
  \floor{\cycle(\V{x})}_{\gtt}=\frac{K(\V{x})-\{\cycle(\V{x})\}_{\gtt}
  \Gamma}{\Gamma}\,.
\end{align}
These are the number of fractional and complete global traversals in
$\cycle(\V{x})$, respectively. To make this more explicit, note
$K(\V{x})$ is the number of vertices that $\cycle(\V{x})$ contains in
the column containing $\V{x}$. Let $(\V{x}_{i})_{i=0}^{K(\V{x})-1}$
denote the sequence of these vertices, starting from
$\V{x}$. Partition $\cycle(\V{x})$ into a sequence of distinct global
traversals $\gt{\V{x}_{j\Gamma}}$ for
$j=0,\dots, \floor{\cycle(\V{x})}_{\gtt}$ together with a sequence of
$\{\cycle(\V{x})\}_{\gtt}$ consecutive strands starting from
$\V{x}_{\floor{\cycle(\V{x})}_{\gtt}\Gamma}$ and ending at
$\V{x}$. This decomposition of $\cycle(\V{x})$ may depend on $\V{x}$,
but this ambiguity will cause no trouble: the key point is that
$\{\cycle(\V{x})\}_{\gtt}$ and $\floor{\cycle(\V{x})}_{\gtt}$ are
independent of $\V{x}$.

\subsubsection{Existence and measurement of global traversals}
\label{sec:first-prop-glob}

\newcommand{\istrand}{\tilde Y}

We will shortly show that each cycle contains at least one global
traversal. This is a key step in reducing the study of cycles to the
study of global traversals. To establish this we consider first
consider long sequences of strands, which allows for a reduction to a
statement about random walks. That is, set
\begin{equation}
  \label{eq:strint}
  Y_{nk+j}(\V{x}) = \sum_{\ell=0}^{k-1}\sti{\V{x}_{\ell}}_{n-1} +
  \sti{\V{x}_{k}}_{j}, \qquad k\in\ccb{\frac{\Gamma}{100D}},\quad j\in\ccb{n},
\end{equation}
so $Y$ is the random walk obtained by following the increments
along $\Gamma/100D$ consecutive strands. Note that $Y$ does not have
independent increments due to the dependence between strands. Let
$\istrand$ be the random walk with IID increments obtained by
concatenating $\Gamma/100D$ independent copies of $\sti{\V{x}}$.

\begin{lemma}
  \label{lem:GTfrac}
  Let $\mathcal{E}$ be the event that $Y=\istrand$ under the optimal
  coupling of $Y$ and $\istrand$. There is an $r>0$ such that
  $\P\cb{\mathcal{E}\mid S^{1} (D)}\geq 1-e^{-r\sqrt{m\log m}}$.
\end{lemma}
\begin{proof}
  We prove we can couple $Y$ and $\tilde Y$ with probability
  $1-Re^{-r\sqrt{m\log m}}$ for some $r,R>0$ on $S^{1} (D)$. On
  $S^{1} (D)$ the maximal upward displacement of $Y$ is
  $\frac{\Gamma}{100D}(4+\frac{3D}{2})\sqrt{m\log m}\leq \frac{m}{2}$,
  so the last strand of $Y$ is at least distance
  $D\sqrt{m\log m}$ from the first strand in every column. Thus
  the total variation distance between an increment of $Y$ and an
  increment of $\istrand$ on $S^{1} (D)$ is at most
  $\exp(-\Theta( \sqrt{m\log m})$ by \Cref{lem:RW-CD}. The
    implicit constant in the $\Theta$ notation depends only on the
    constants $c_{1},c_{2}$ from \Cref{lem:RW-CD}. A union bound
  over the $\frac{\Gamma}{100D}n\leq 2m^{3/2}$ increments gives the
  result.
\end{proof}

\begin{lemma}
  \label{lem:GTnD}
  There are $R,r>0$ such that
  \begin{equation}
    \label{eq:GTnD}
    \P\cb{\text{exists $\V{x}$ such that
        $\floor{\cycle(\V{x})}_{\gtt}=0$}| S^{1} (D)}\leq Re^{-r(\log m)^{3/2}}.
  \end{equation}
  Moreover, if $\V{x}'$ denotes the first vertex of the last strand of
  $\gt{\V{x}}$, then we further have that
  \begin{equation}
    \label{eq:GTnD1}
    m^{3/4}(\log m)^{1/2}(1+o(1))\leq d(\st{\V{x}},\st{\V{x}'})\leq 3
    m^{3/4}(\log m)^{1/2}
  \end{equation}
  with probability $1-Re^{-r(\log m)^{3/2}}$.
\end{lemma}
\begin{proof}
  For convenience set $\V{x}=(0,0)$. For the first statement, we must
  show that the cycle beginning at $\V{x}$ will not be explored in
  fewer than $\Gamma$ strands. We prove this by showing that
  $0\leq |\pi_{2}(\V{x}_{j})| \leq m-m^{1/4}C(m)$ for all
  $j\in \ccb{0,\Gamma}$ with the requisite probability, $\V{x}_{j}$
  representing the initial vertices of strands. On $S^{1} (D)$
  consecutive strands are at vertical distance of order
  $\sqrt{m\log m}$, so this will suffice.
  
  Decompose $\cycle(\V{x})$ into sequences of $\Gamma/100D$
  consecutive strands, and consider the first $100D$ such
  sequences. Write $W$ for such a sequence of $\Gamma/100D$
  strands. We will show that the collective fluctuations of the $W$
  cannot overcome the vertical gap created by the $-2m^{1/4}$ in the
  definition of $\Gamma$.

  To do this, treat each $W$ as a deterministic part plus $Y$, the
  random walk formed by the increments of $W$. Each strand has a
  deterministic vertical increase of $n=m+C(m)$, and hence a vertical
  displacement of $C(m)$. 
  By displacement we mean
  $|\pi_{2}(\V{x}_{j+1})-\pi_{2}(\V{x}_{j})|$. The deterministic
  vertical displacement from the first 100D sequences is therefore
  \begin{equation*}
    100D\frac{\Gamma C(m)}{100D} = \Gamma C(m) = m -
    2m^{1/4}C(m)
  \end{equation*}
  A net fluctuation of size $2m^{1/4}C(m)$ is required for the first
  and last strands to meet. We will rule out a fluctuation of size
  $m^{1/4}C(m)$.
  
  By Lemma~\ref{lem:GTfrac} the fluctuation due to each sequence of
  strands can be treated as $\istrand$, a simple random walk of the
  same length. This replacement results in an error with probability
  at most $\exp(-r\sqrt{m\log m})$. Hence by a union bound, the net
  fluctuation in the first 100D sequences can be treated as that of a
  concatenation of 100D copies of $\istrand$, i.e., a simple random
  walk of length $\Gamma n = (1+o(1))\Gamma m$. Let
  \begin{equation*}
    x = \frac{ m^{1/4}C(m) }{(\Gamma m)^{1/2}} \geq \frac{ m^{1/4}
      (C(m))^{3/2} }{m} \geq (\log m)^{3/4}.
  \end{equation*}
  By \cite[Theorem~5.23]{Petrov} (recall~\eqref{eq:Petrov}), the
  probability of a fluctuation of size $x$ is of order
  $\exp(-\Theta(x^{2})) = \exp(-\Theta ( (\log m)^{3/2}))$. This
  proves the first claim.
  For the second, note that when a fluctuation of size at most
  $x$ occurs the remaining gap is of size at least
  $(1+o(1))m^{1/4}C(m)$ and at most $3m^{1/4}C(m)$.
\end{proof}

The next lemma says that $\floor{\cycle}_{\gtt}$ is a good proxy for
the size of a cycle. 
\begin{lemma}
  \label{lem:GTintegral}
  Write $\cup_{\cycle}$ for the union over all cycles. There is an $r>0$ such that
  \begin{equation}
    \label{eq:GTintegral}
    \P\cb{ \bigcup_{\cycle} \Bigg\{
      \frac{\{\cycle\}_{\gtt}}{\floor{\cycle}_{\gtt}}\geq m^{-1/8}
      \Bigg\} \Bigg|  S^{1} (D)} \leq
    \exp(-r \sqrt{m\log m}).
  \end{equation}
\end{lemma}
\begin{proof}
  We work on the event that~\eqref{eq:GTnD1} holds for all global
  traversals. By a union bound over all $\V{x}$ and \Cref{lem:GTnD},
  there is an $r>0$ such that this event has probability at least
  $1-\exp(-r \sqrt{m\log m})$.

  The proof of \eqref{eq:GTintegral} proceeds by a case analysis. By a
  union bound it is enough to consider a single cycle, as there are at
  most $m$ cycles. Suppose a cycle $\cycle$ has exactly
  $k=\floor{\cycle}_{\gtt}\leq m^{1/8}$ completed global
  traversals. Concatenate these completed global traversals. Equation
  \eqref{eq:GTnD1} implies the distance between the endpoint and
  initial point of the concatenation is of order
  $km^{3/4}(\log m)^{1/2}$. The number of strands that connect these
  points is of order $km^{3/4}(\log m)^{1/2}/C(m)= \Theta(km^{1/4})$,
  as each strand has a deterministic contribution of $C(m)$ plus a
  random fluctuation of at most $(4+\frac{3D}{2})\sqrt{m\log m}$. Thus
  the fraction of strands in the fractional part is of order
  \begin{equation*}
    \{\cycle\}_{\gtt}/\floor{\cycle}_{\gtt} =
    \Theta(\frac{km^{1/4}}{k\Gamma}) = \Theta(\frac{m^{1/4}(\log
      m)^{1/2}}{m^{1/2}}) 
    \leq \frac{1}{m^{1/8}}.
  \end{equation*}
  On the other hand, if $\cycle$ has exactly
  $k=\floor{\cycle}_{\gtt}> m^{1/8}$ completed global
  traversals, the claim follows since a fractional traversal consists of less than one
  global traversal.
\end{proof}

Let $|\cycle|_{\gtt}=\floor{\cycle}_{\gtt} +
\{\cycle\}_{\gtt}$. \Cref{lem:GTintegral} shows that
\begin{equation}
  \label{eq:SysSize}
  m = \Gamma\sum_{\cycle}|\cycle|_{\gtt} =
  \Gamma(1+O(m^{-1/8}))\sum_{\cycle}\floor{\cycle}_{\gtt}
\end{equation}
with high probability on $S^{1} (D)$. This says that the number of
global traversals in cycles is within $o(1)$ of the maximal number
$m/\Gamma$ of possible global traversals in the system.

\begin{remark}
  \label{rem:Sdredef}
  Henceforth we define $S(D)$ to be the sub-event of $S^{1}(D)$ on
  which the events in \Cref{lem:GTnD,lem:GTintegral} occur. Note that
  $\P\cb{S(D)} = 1-o(m^{-D/2})$.
\end{remark}

\section{Equilibrium properties II: Concentration of Contacts}
\label{sec:Contacts}

This section studies pairs of global traversals, and establishes that
the number of times they come into contact is a concentrated random
variable. We make the notion of a contact precise in
Section~\ref{sec:Contacts-Defn} and formalize concentration in
Theorem~\ref{thm:CofC}. Subsequent subsections prove
Theorem~\ref{thm:CofC}.

\subsection{Definition of contacts, global concentration}
\label{sec:Contacts-Defn}

Recall the notions of swapping and parallel arrows from
Section~\ref{sec:No-shifts}. Given $v\in V(\T^{\star}_{n,m})$,
$v= \{(x_{1},x_{2}),(x_{1},x_{2}+1)\}$ (so $v$ is a vertical edge of
$\T_{n,m}$), write $v\pm\frac12$ for $(x_{1},x_{2})$ and
$(x_{1},x_{2}+1)$.
\begin{definition}
  \label{def:contact}
  Let $v\in V(\T^{\star}_{n,m})$ be a vertical edge. There is a
  \emph{contact} at $v$ \emph{between} $v+\frac12$ and $v-\frac12$ if
  the arrows $\phi_{v\pm\frac 12}$ swap or are parallel.
\end{definition}
Contacts are the locations where the Glauber dynamics can modify the
corresponding hard-core configuration. Glauber updates at these
locations lead to cycles splitting or merging, recall
Remark~\ref{rem:Glauber-SM}.

Given $A,B\subset V(\T_{n,m})$ and a realization of
$\Phi$, the \emph{number of contacts between $A$ and $B$} is the number of
pairs $(a,b)\in A\times B$ such that there is a contact between $a$
and $b$. Given a set $A\subset V(\T_{n,m})$, the \emph{number of contacts
contained in $A$} is the number of contacts between $A$ and
$V(\T_{n,m})$. 
\begin{lemma}
  \label{lem:Strand-Contacts}
  Let $X$ be the number of contacts contained in a strand. There
  exist $c,\alpha>0$ such that 
  \begin{equation*}
    \P\cb{ \abs{X-\alpha n}\geq \gamma \alpha n\mid S(D)} \leq
    e^{-c\gamma^{2}n}, \qquad \gamma>0.
  \end{equation*}
\end{lemma}
\begin{proof}
  Recall that on $S(D)$, we can use the underlying hard-core models on
  columns of $\T^{\star}_{n,m}$ as the source of randomness. A single
  strand is determined by $n$ IID samples from the equilibrium measure
  of the hard-core model. Each sample results in $i$ contacts with
  probability $p_{i}$ for $i=0,1,2$, $\sum_{i}p_{i}=1$,
  $\sum_{i}ip_{i}=\alpha$, and $\alpha>0$ since
  $\lambda>0$. Hoeffding's inequality then implies the claim.
\end{proof}

The previous lemma indicated that contacts are plentiful. The next
theorem indicates that they are rather uniformly distributed between
pairs of disjoint global traversals. The proof of the theorem
occupies \Cref{sec:IGC,sec:tbd}. The details of the proof will not be
needed for subsequent developments.
\begin{theorem}
  \label{thm:CofC}
  Let $\V{x}\neq \V{y}\in \T_{n,m}$. Let $X$ denote the number of
  contacts between $\gt{\V{x}}$ and $\gt{\V{y}}$ conditional on
  $\gt{\V{x}}\cap\gt{\V{y}}=\emptyset$. Then
  \begin{equation}
    \label{eq:CofC}
    \P\cb{ |X-\E X|\geq m^{1/3}\Big| S(D)}
    \leq \exp\ob{-\Theta(m^{1/7})}, \quad \E X =
    \Theta\ob{\frac{m}{\log m}}.
\end{equation}
\end{theorem}

\subsection{The ideal gap chain}
\label{sec:IGC}

The \emph{ideal gap chain} $Z\colon\nwithzero \to \nwithoutzero$ is a
time-homogeneous Markov chain that offers an idealized description of
the vertical distance between two strands. In defining $Z$ we make use
of the hard-core model on $\Z+1/2$ as introduced in
\Cref{sec:HC-CD}. The shift by $1/2$ is made with
Figure~\ref{fig:Dual} and \Cref{rem:HC-Perm} in mind: each hard-core
configuration defines a corresponding arrow configuration $\psi$ on~$\Z$. The trajectories determined by a sequence of these arrow
configurations will be our idealization of strands, and $Z$ the
distance between two such strands. We now make this precise.

\begin{definition}
  \label{def:IGC}
  Let $\sigma$ denote the hard-core model on $\Z + 1/2$ with activity
  $\lambda$ given by~\eqref{eq:HCactivity}.  Let
  $p=\P\cb{\sigma_{1/2}=1}\in (0,1/2)$ denote the probability of any
  particular vertex being occupied.  Let
  $q_i(a) = \P \cb{\sigma_{i+1/2}=1 \vert \sigma_{a}=1}$ for
  $a \in \{-1/2,1/2 \}$, and
  $q_i(0) = \P \cb{\sigma_{i+1/2}=1 \vert \sigma_{-1/2}=\sigma_{1/2}=0}$.
  The \emph{ideal gap chain $Z\colon \N\to\nwithoutzero$} is the
  time-homogeneous Markov chain with transition probabilities
  ($i\geq 2$)
\begin{equation*}
        \P \cb{ Z_{1} - Z_{0} = j \big\vert Z_{0} = i } 
=
   \begin{cases}
    \, \, p q_i(-1/2)   & j=2 \\
    \,    \, (1-2p) q_i(0)  + p \big( 1 - q_i(-1/2) - q_{i-1}(-1/2) \big) & j=1 \\
      \,  \, (1-2p) \big(1 - q_i(0) - q_{i-1}(0)\big) +
      p\big(q_i(1/2)+ q_{i-1}(-1/2)\big)  & j=0 \\ 
      \,  \,  (1-2p) q_{i-1}(0)  + p \big( 1 - q_i(1/2) - q_{i-1}(1/2) \big) & j=-1 \\
      \,  \, p q_{i-1}(1/2)  & j=-2
    \end{cases}.
\end{equation*}
The last possibility does not occur if $i=2$ since $q_{1}(1/2)=0$. If
$i=1$ the transition probabilities are 
\begin{equation*}
        \P \cb{ Z_{1} - Z_{0} = j \big\vert Z_{0} = 1}
= 
   \begin{cases}
    \, \, p q_1(-1/2)   & j=2 \\
    \,    \, (1-2p) q_1(0)  + p \big( 1 - q_1(-1/2) - q_{0}(-1/2) \big) & j=1 \\
    \,  \, (1-2p) \big(1 - q_1(0)\big) + p  & j=0 
    \end{cases}.
\end{equation*}
\end{definition}

Let $\psi=(\psi_{i})_{i\in\Z}$ denote the arrow configuration on $\Z$
associated to the hard-core model on $\Z+\frac12$, and let
$\Psi = (\psi^{j})_{j\in\N}$ denote a sequence of independent copies
of $\psi$. Thus $\psi^{j}_{i}\in \{-1,0,1\}$ is the arrow at $i\in\Z$
in the $j$\textsuperscript{th} copy of $\psi$, and for
$(i,j)\in \Z\times\N$, $\Psi((i,j))=\psi^{j}_{i}$.  Set
$\Psi^{0}(i)=i$, and for $j\geq 0$ set
$\Psi^{j+1}(i)=\Psi^{j}(i) + \psi^{j}_{\Psi^{j}(i)}$.  Thus
$(\Psi^{j}(i))_{j\in\nwithzero}$ is the analogue of a strand.  More
precisely, this is the analogue of $\sti{\V{x}}$ (recall
\eqref{eq:increment-explicit}), but the omission of the deterministic
$+1$ will be irrelevant since we will be considering the distance
between two strands.

\begin{lemma}
  \label{lem:idealgap}
  Let $i \in \nwithoutzero$. Let
  $X\colon \nwithzero \to \nwithoutzero$ be given by
  $X_{j}=\big\vert \Psi^{j}(i)-\Psi^{j}(0) \big\vert$.  Then
  $X$ is equal in law to the ideal gap chain~$Z$ with $Z_{0}=i$.
\end{lemma}
\begin{proof}
  We consider $i\geq 2$; the case $i=1$ is similar and we omit the
  details. Since there is an independent hard core model $\sigma$ for
  each $j\in \nwithzero$, it is enough to verify that $X_{1}-X_{0}$ is
  distributed as $Z_{1}-Z_{0}$.

  Flip a three-sided coin with outcomes of probability $p$, $p$ and
  $1-2p$ called heads, tails and blank. If it is heads, condition
  $\sigma$ to contain $1/2$. Let $q_j(1/2)$ denote the conditional
  probability that $j + 1/2$ is in $\sigma$.  Then the conditional
  distribution of $Z_{1}-Z_{0}$ equals $0$ with probability
  $q_i(1/2)$; $-1$ with probability $1 - q_i(1/2) - q_{i-1}(1/2)$; and
  $-2$ with probability $q_{i-1}(1/2)$.

  If the outcome is tails, condition $\sigma$ to contain $-1/2$. Let
  $q_j(-1/2)$ denote the conditional probability that $j + 1/2$ is in
  $\sigma$.  Then the conditional distribution of $Z_{1}-Z_{0}$ equals
  $2$ with probability $q_i(-1/2)$; $1$ with probability
  $1 - q_i(-1/2) - q_{i-1}(-1/2)$; and $0$ with probability
  $q_{i-1}(-1/2)$.

  If the outcome is blank, condition $\sigma$ to contain neither $1/2$
  nor $-1/2$. Let $q_j(0)$ denote the conditional probability that
  $j + 1/2$ is in $\sigma$.  Then the conditional distribution of
  $Z_{1}-Z_{0}$ equals $1$ with probability $q_i(0)$; $0$ with
  probability $1 - q_i(0) - q_{i-1}(0)$; and $-1$ with probability
  $q_{i-1}(0)$.
\end{proof}

\begin{lemma}
  \label{lem:gapchainrep}
 There are independent events $(S_{k})_{k\in\nwithzero}$ such that the
 ideal gap chain satisfies
 \begin{equation}
   \label{eq:gapchainrep}
  Z_{k+1} - Z_k
 = ( \tilde X_k +  \tilde X_k') {\bf 1}_{S_k} + E_k {\bf 1}_{S_k^c} \, , 
\end{equation}
 where
 \begin{enumerate}
 \item $\tilde X_k, \tilde X_k'$, $k \in \nwithzero$, are independent random
   variables that share a symmetric and non-degenerate
   distribution~$\mu$ taking values in $\{-1,0,1\}$;
 \item $E_k$, $k \in \nwithzero$, are independent random variables
   taking values in $\llbracket -2,2 \rrbracket$; and
 \item there are constants $c,C>0$ such that the events $S_j$ satisfy,
   for $k,k',j \in \nwithzero$,
 \begin{equation*}
 \P \cb{S_k^c \big\vert Z_{k}
 = j }
 =
 \P \cb{ S_{k'}^c \big\vert Z_{k'}
 = j } \leq C e^{-cj} \, \, \, \textrm{for $j \in \nwithzero$} \,.
\end{equation*}
\end{enumerate}
\end{lemma}
\begin{proof}
  \Cref{lem:HC-CD} implies that there exist positive constants $C$ and
  $c$ such that $\vert q_i(a) - p \vert \leq C e^{-ci}$ for
  $a \in \{-1/2,0,1/2 \}$. Items (1)--(3) of \Cref{lem:gapchainrep} then
  follow readily from \Cref{lem:idealgap}. 
\end{proof}

For $j \in \nwithoutzero$ let $\tau_j = \min \{ k \in \nwithzero \mid
Z_{k} = j \}$ denote the hitting time of $j$ by~$Z$, and $\tau_{j}^{+}
= \min \{ k\in\nwithzero \mid Z_{k}\geq j\}$. The next
proposition says that the ideal gap chain behaves like a
simple random walk. 

\begin{proposition}
  \label{prop:idealgapestimates}
  For $j_1,j_2 \in \nwithoutzero$ with $j_2 \geq j_1>1$, 
  \begin{equation}
    \label{eq:ideal1}
    \P\cb{ \tau^{+}_{j_2} < \tau_1 \, \big\vert \, Z_{0} = j_1 } = \Theta \big( j_1 j_2^{-1} \big).
  \end{equation}
\end{proposition}
\begin{proof}
  We begin by deriving the upper bound, i.e., showing there exists a constant $C > 0$ for which
  \begin{equation}\label{e.ubtau}
    \P \cb{ \tau^{+}_{j_2} < \tau_1 \, \big\vert \, Z_{0} = j_1 }
    \, \leq \, C j_1 j_2^{-1} \, . 
  \end{equation}
  We will do this in three steps. First, we will construct a process
  $(Z'_{k})_{k\geq 0}$ that stochastically dominates
  $(Z_{k})_{k\geq 0}$. Informally, $Z'$ is the process that results
  from forcing $E_{k}=2$ in~\eqref{eq:gapchainrep}. Secondly, we will
  bound $Z'$ above by a Markov chain that can be analyzed by standard
  methods. Third, we will use this Markov chain to push $Z'$ down,
  which will imply that $Z$ visits $1$.

  For the first step, define
  \begin{equation*}
    Z'_{k+1}-Z'_{k} = (X_{k}+X_{k}'){\bf 1}_{S'_{k}} + 2\cdot {\bf 1}_{(S'_{k})^{c}},
  \end{equation*}
  where the $S'_{k}$ are independent events, with
  $\P \cb{ (S'_{k})^{c} \mid Z'_{k}=j} = \max_{i\in\ccb{0,3}} \P \cb{
    S_{k}^{c} \mid Z_{k}=j-i}$; in this definition we set
  $\P\cb{S_{k}^{c} \mid Z_{k}=i}=0$ for $i\leq 0$. Set $Z'_{0}=Z_{0}$,
  and suppose that $Z'_{k}\geq Z_{k}$. If $Z'_{k}-Z_{k}\geq 4$, then
  $Z'_{k+1}\geq Z_{k+1}$ since each process jumps by at most two. On
  the other hand, if $Z'_{k}-Z_{k}\in \ccb{0,3}$, then we may couple
  $Z'_{k+1}$ with $Z_{k+1}$ such that $S_{k}^{c}\subset (S'_{k})^{c}$,
  and hence $Z'_{k+1}\geq Z_{k+1}$.
  
  For the second step, note that $(Z'_{k})_{k\geq 0}$ with
  $Z'_{0}=j_{1}$ is equal in law to the following auxiliary process
  $(Y_{k})_{k\in \N}$ with $Y_{0}=j_{1}$, observed only at even times
  $k\in 2\nwithzero$.  Let $S(k,j)$ be events such that
  $\P\cb{S(k,j)} = \P\cb{S_{k}|Z_{k}=j}$, independent for different
  $k$. Suppose we are given the trajectory $(Y_i)_{i=0}^{2k}$.  If
  $S(2k,Y_{2k})$ occurs, set $Y_{2k+1} = Y_{2k} + X$ and
  $Y_{2k+2} = Y_{2k+1} + X'$, where $X$ and $X'$ are independent
  random variables with the symmetric law $\mu$ on $\{-1,0,1\}$
  from \Cref{lem:gapchainrep}.  If $S(2k,Y_{2k})^c$ occurs, set
  $Y_{2k+1} = Y_{2k}+1$ and $Y_{2k+2} = Y_{2k+1}+1$.

  The process $Y$ is not Markov, but we may bound it above
  stochastically by a Markov process
  $\yplus\colon \nwithzero \to \nwithoutzero$, $\yplus_0 =j_{1}$.  We
  first explain this informally. The process $Y$ flips a coin that
  lands heads with probability $\P\cb{S(2k,Y_{2k})}$ at even times,
  takes two steps $X,X'$ if the coin lands heads, and takes two steps
  $+1$ otherwise.  $\yplus$ will instead flip a coin at each time,
  taking an independent step distributed as $X$ for heads, and $+1$
  for tails. To obtain the desired stochastic domination it suffices
  to show that we can arrange that if $Y$ gets a tail at time $2k$,
  then $\yplus$ gets a tail at times $2k$ and $2k+1$, as if $Y$ gets a
  head the increments of $Y$ are dominated by those of $\yplus$.

  To achieve this, define
  $p^{+}_{k} =\max_{j\in \ccb{0,2}}
  \sqrt{\P\cb{S(k,\yplus_{k}-j)^{c}}}$. In this formula,
  $\P\cb{S(k,i)^{c}}=0$ if $i\leq 0$ by convention.  We make the
  $\yplus$ coins land heads with probabilities $1-p^{+}_{k}$, and
  tails with the complementary probability $p^{+}_{k}$. With this
  definition $\yplus$ is Markov as its transition probabilities depend
  only on $\yplus_k$, and we can couple the coins as desired.

  We proceed to the third step. Note that \Cref{lem:gapchainrep} implies
  \begin{equation*}
    \P \cb{ \yplus_{k+1} - \yplus_k = j \big\vert \yplus_k = \ell }
    =  
    \begin{cases}
    \, \, \mu(1) + \nu_1(\ell) & \textrm{for $j=1$} \\ 
     \, \, \mu(0) + \nu_2(\ell) & \textrm{for $j=0$} \\  
     \, \, \mu(1) + \nu_3(\ell) & \textrm{for $j=-1$}          
   \end{cases} \, ,
 \end{equation*}
 where $\vert \nu_i(\ell) \vert \leq C \exp \{ - c \ell \}$ for
 $i \in \intint{3}$. Choose $a$ such that $\vert \nu_i(a') \vert <1$
 for all $a'\geq a$ and $i\in\ccb{1,3}$. When $\yplus\geq a$, $\yplus$
 has the transitions of a birth-and-death chain (see,
 e.g.~\cite[Example~1.3.4]{Norris}). A standard computation, which
 we give below in \Cref{lem:vestimate}, shows the hitting times of
 $\yplus$ satisfy
 \begin{equation}
   \label{eq:yplusa}
   \P^{\yplus} \cb{\tau_{j_2}^{+} < \tau_a \Big\vert \yplus_0 =
     j_1  } \, \leq \, C j_1 j_2^{-1} \, 
 \end{equation}
 for some $C>0$. We now deduce~\eqref{e.ubtau}
 from~\eqref{eq:yplusa}. Since $Y\leq \yplus$ almost surely,
 $Y_{0}=\yplus_{0}=j_{1}$, and $\yplus$ takes jumps of size at most
 one, the equality in law of $Y$ and $Z'$ at even times implies
 \eqref{eq:yplusa} also holds for the process $Z'$ if $\tau_{a}$ is
 replaced by the hitting time of $\{a,a+1\}$. Since $Z\leq Z'$ almost
 surely, this implies the same statement for $Z$. Finally, the ideal
 gap chain has a positive probability of hitting $1$ before $j_{2}$ if
 started at $a$ or $a+1$, as can be seen by constructing an explicit
 trajectory. This completes the proof of~\eqref{e.ubtau}.

   To complete the proof of Proposition~\ref{prop:idealgapestimates}
   we must prove the complementary lower bound, that there exists
   $c > 0$ such that
   $\P \cb{ \tau_{j_2}^{+} < \tau_1 \, \big\vert \, Z(0) = j_1 } \geq
   c j_1 j_2^{-1}$. The argument is very similar to the one above, and
   so we only outline the steps. First, we repeat the above construction,
   but define $Z'$ by setting $E_{n}=-2$ and taking the maximum of
   probabilities when $Z_{k}=Z_{k}'+i$ for $i\in \ccb{0,3}$ (if such a
   jump would result in $Z'_{k}<1$, we set $Z'_{k}=1$). Second,
   introduce $\yminus\colon \nwithzero \to \nwithoutzero$ such that
   $Y \geq \yminus$ by following the specification of $\yplus$, but
   now replacing the $+1$ steps on tails with $-1$ steps, and again
   modifying the maximum. Third, the construction of an explicit
   trajectory to have $Z$ hit $1$ is replaced by the construction of
   an explicit trajectory to get $Z$ to a state $a$ above which the process
   $\yminus$ is a birth-and-death process.  The analogue of
   \Cref{lem:vestimate} for $\yminus$ then yields the sought bound.
\end{proof}

\begin{lemma}
  \label{lem:vestimate}
  In the context of the proof of \Cref{prop:idealgapestimates}, for
  $a<j_{1}<j_{2}$ we have that
  \begin{equation*}
    \P^{\yplus} \cb{ \tau_{j_2}^{+} < \tau_a \Big\vert \yplus_0 = j_1
    } = \Theta( j_1 j_2^{-1}) \, . 
  \end{equation*}
\end{lemma}
\begin{proof}
Let $\big\{ \kappa_i : i \in \nwithzero \big\}$ be the increasing
sequence given by $\kappa_{a-1} = 0$, $\kappa_a = 1$, and
\begin{equation}
  \label{e:kappa}
 \frac{\kappa_{\ell +1} - \kappa_\ell}{\kappa_\ell - \kappa_{\ell - 1}} \, = \, \frac{\mu(1) + \nu_1(\ell)}{\mu(1) + \nu_3(\ell)} \, \, \, \textrm{for $\ell \geq a$} \, .
\end{equation}
Then (as $\tau_{j_{2}}^{+}=\tau_{j_{2}}$ and $Y^{+}$ is a standard birth-death chain, see, e.g.~\cite[Example~1.3.4]{Norris})
\begin{equation}
  \label{e:oneoverkappa}
 \P^{\yplus} \cb{ \tau_{j_2}^{+}< \tau_{a} \Big\vert \yplus_0 = j_1 } \, = \, \frac{\kappa_{j_1}}{\kappa_{j_2}} \, .
\end{equation}
Note that 
\begin{equation*}
  \frac{\kappa_{\ell +1} - \kappa_\ell}{\kappa_\ell - \kappa_{\ell -1}}
  = \big( 1 + \nu_1(\ell)\mu(1)^{-1} \big) \big( 1 + \nu_3(\ell)\mu(1)^{-1} \big)^{-1}
  = 1 + \Theta(e^{-c \ell}) \, . 
\end{equation*}
Hence, there exist constants $K \in (0,\infty)$ and  $\chi \in (0,\infty)$ such that 
\begin{equation}
  \label{e:kappalocations}
 \kappa_\ell = \ell \chi + K + \Theta(e^{-c \ell}) \, ,
\end{equation}
and inserting this into~\eqref{e:oneoverkappa} yields the lemma.
\end{proof}

Let $a\in\nwithzero$ be the state introduced above \eqref{eq:yplusa},
i.e., the state such that $Y^{\pm}$ is a birth-and-death chain when at
or above state $a$. The next lemma provides couplings of the bounding
Markov chains $Y^{\pm}$ with Brownian motions when $Y^{\pm}\geq a$;
these couplings will be useful for analysing hitting times of $Z$ in
the next section. More precisely, to circumvent the restriction
$Y^{\pm}\geq a$, let $\yypm$ denote $Y^{\pm}$ with jumps below level
$a$ suppressed (i.e., nothing happens if such a jump is attempted),
and similarly, jumps of size $0$ suppressed.  The couplings require some
notation. Write $\{\kappa^{\pm}_{i}\}_{i\geq 1}$ for the
$\kappa^{\pm}_{i}$ given by \eqref{e:kappa} with
$\kappa^{\pm}_{a-1}=0$, $\kappa^{\pm}_{a}=1$. The distinction between
$\pm$ arises from the meaning of $\mu$ and $\nu$ in
\eqref{e:kappa}. Set
\begin{equation}
  \label{eq:kappaN}
  \bar\kappa^{\pm}_{a-i}=2\kappa^{\pm}_{a}-\kappa^{\pm}_{a+i}, \quad i\geq 2,
  \qquad \text{and}\quad  \mathcal{K}^{\pm}= \big\{\kappa^{\pm}_{i}\big\}_{i\geq a}\bigcup
\big\{\bar\kappa^{\pm}_{a-i}\big\}_{i\geq 1}. 
\end{equation}
\begin{lemma}
  \label{lem:BCouple}
  There is a coupling of $\yyplus$ 
  with a Brownian motion $B\colon [0,\infty)\to \R$ such that
  \begin{enumerate}
  \item There is an increasing sequence of stopping times
    $\rho^{+}_{i}$ with $\rho^{+}_{0}=0$ such that
    $B(\rho^{+}_{i})\in \mathcal{K}^{+}$.
  \item For each $i$, $\arg B(\rho^{+}_{i})=\yyplus_{i}$,
    where $\arg
    x=a+i$ if $x=\kappa^{+}_{a+i}\in \mathcal{K}^{+}$ or $x=\bar\kappa^+_{a-i}$.
  \item There are $\beta,c>0$ such that if $b>0$ and $E^{+}(i,b)$ is
    the event that $\rho^{+}_{i}-\beta i>bi$, then
    \begin{equation}
      \label{eq:BCouple1}
      \P\cb{E^{+}(i,b)}\leq e^{-cbi}.
    \end{equation}
  \end{enumerate}
  An analogous coupling of $\yyminus$ 
  with a Brownian motion exists,
  with $\mathcal{K}^{+}$ replaced with $\mathcal{K}^{-}$.
\end{lemma}
\begin{proof}
  Suppose $\yyplus_{0}=j\in\nwithoutzero$, $j\geq a$, and set
  $B(0)=\kappa_{j}$. We will construct the couplings for $\yyplus$ and
  $\yyminus$ in parallel (i.e., one should consistently choose the sign
  $+$ or the sign $-$). We first give the construction up to the first
  time $\tau$ that $\yypm$ 
  hits $1$, and afterward we extend the
  construction to all times. To this end, iteratively define an
  increasing sequence of times
  $\big\{ \rho^\pm_i: i \in \nwithzero \big\}$ such that
  $\P\cb{\yypm_{i+1}=k+1 \mid \yypm_{i}=k}$ is equal to
  $\P\cb{B(\rho^{\pm}_{i+1})=\kappa^{\pm}_{k+1} \mid
    B(\rho^{\pm}_{i})=\kappa^{\pm}_{k}}$. Recall $\rho^{\pm}_0 = 0$,
  and suppose an initial sequence $\rho^{\pm}_j$,
  $j \in \intinta{0}{i-1}$ has been specified such that for
  $i \in \nwithoutzero$, $B(\rho^{\pm}_j) \in \mathcal{K}^{\pm}$ for
  $j \in \intinta{0}{i-1}$. We now define $\rho^{\pm}_i$ to be the
  infimum of $t \geq \rho^{\pm}_{i-1}$ such that
  $B(t) \in \mathcal{K}^{\pm}$ with
  $B(t) \not= B\big(\rho^{\pm}_{i-1} \big)$.  With these definitions,
  since the $\kappa_{i}^{\pm}$ form an increasing sequence,
  \begin{equation}
    \label{eq:coupling}
      \frac{\P\cb{B(\rho^{\pm}_{i+1})=\kappa_{j+1}| B(\rho^{\pm}_{i})=\kappa^{\pm}_{j}}}
      {\P\cb{B(\rho^{\pm}_{i+1})=\kappa^{\pm}_{j-1}| B(\rho^{\pm}_{i})=\kappa^{\pm}_{j}}}
      = \frac{\P\cb{\yypm_{i+1}=j+1 | \yypm_{i}=j}}
      {\P\cb{\yypm_{i+1}=j-1 | \yypm_{i}=j}}
      =
      \frac{\kappa^{\pm}_{j+1}-\kappa^{\pm}_{j}}{\kappa^{\pm}_{j}-\kappa^{\pm}_{j-1}}.
    \end{equation}
    This gives our coupling up until the hitting time $\tau$ of $a$ by
    $\yypm$: $(\yypm_{i})_{i\in\nwithzero}$ is equal in law to
    $(\arg B(\rho^{\pm}_{i}))_{i\in\nwithzero}$.

    To extend the coupling to all times, note that
    $B(\rho^{\pm}_{\tau+1})$ is equally likely to be
    $\bar\kappa^\pm_{a-1}$ and $\kappa_{a+1}$. The distances between
    the $\bar\kappa^\pm_{a-i}$ are exactly as for the
    $\kappa^+_{a+i}$, and hence the coupling can be continued by
    making some small changes in indexing (the Brownian motion going
    to the left corresponds to $\yypm$ going right when at the
    $\bar\kappa$).  This establishes (1) and (2).

    To establish (3), note that the law of
    $\rho^\pm_{i+1} - \rho^\pm_i$ is a mixture of the laws of hitting
    times of Brownian motion begun at one point in $\mathcal{K}^{\pm}$
    on the set of adjacent points in $\mathcal{K}^{\pm}$.  By
    \eqref{e:kappalocations} and \eqref{eq:kappaN} the distances
    between adjacent values of $\kappa^{\pm}$ lie in a compact subset
    of $(0,\infty)$.  Hence by using the reflection principle to
    compute $\P\cb{H<x}$, each such hitting time $H$ satisfies
    $\P\cb{H > x} \leq C e^{-cx}$ for suitable $c,C>0$, and these
    constants are independent of the value of $i\in\nwithzero$. Hence
    $\P\cb{\rho^\pm_{i+1} - \rho^\pm_{i}>x}\leq Ce^{-cx}$. This
    implies that there is a $\beta\in (0,\infty)$ such that
  \begin{equation}
    \label{e:Eprobs}
    \P \cb{ \rho^\pm_{i} - \beta i > b i } 
    \leq C \exp \big\{ - c b i \big\}
  \end{equation}
  for $i \in \nwithoutzero$ and $b > 0$; the values of $c,C>0$ may
  have changed. The deduction of this bound is by writing
  $\rho^{\pm}_{i+1}= \sum_{j=0}^{i}(\rho^{\pm}_{j+1}-\rho^{\pm}_{j})$,
  dominating the increments by a sequence of IID sub-exponential
  random variables, and applying Bernstein's inequality for
  sub-exponential random variables~\cite[Theorem~2.8.1]{Vershynin}.
\end{proof}

\subsection{Concentration of Contacts}
\label{sec:tbd}

This section completes the proof of \Cref{thm:CofC}. The strategy is
to use the estimates of \Cref{prop:idealgapestimates}, which show that
the ideal gap chain behaves like simple random walk, to implement the
heuristic sketched in \Cref{sec:outline}. There are three
steps. \Cref{sec:expl} introduces an exploration procedure that
measures the distance between two global traversals. \Cref{sec:exc}
estimates how long it takes this exploration procedure to find
contacts, and \Cref{sec:cons} uses these estimates and concentration
inequalities to establish \Cref{thm:CofC}.

To lighten notation in this section we often omit rounding to nearest
integer values from the notation, e.g., writing $\sqrt{m}$ in place of
$\floor{\sqrt{m}}$.

\subsubsection{Exploration Procedure}
\label{sec:expl}
To analyse contacts between two global traversals we introduce an
exploration procedure.  Set $\V{x}=(0,x)$, $\V{y}=(0,y)$, $x\neq y$.
The choice of column $0$ entails no loss of generality, and similarly
we will assume $x<y$, i.e., $0<y-x\leq m-(y-x)$. Define
$\Phi^{k}_{2}(\V{u}) = \pi_{2}(\Phi^{k}(\V{u}))$, and (abusing
notation) write $\Phi^{k}_{2}(x)$ in place of $\Phi^{k}_{2}(\V{x})$
when $\V{x}=(0,x)$. Set
\begin{equation}
  \label{eq:ExplPre}
  \hat Z_{s,\ell} = \Phi^{s}_{2}(y) - \Phi^{\ell}_{2}(x).
\end{equation}
Note that if $s=\ell \!\!\!\mod n$, $\hat Z_{s,\ell}$ measures the vertical distance
between two strands of $\gt{\V{x}}$ and $\gt{\V{y}}$. 
Our exploration procedure uses $\hat Z_{s,\ell}$, with $s=\ell=0$ initially. We
first describe the procedure informally. Let $\eta'=\frac{D}{2}$, and
$\eta(m)=\eta'\sqrt{m\log m}$; $\eta(m)$ is a scale at which
global traversals are (relatively) close. When the global traversals
are close, the exploration procedure will examine two strands
simultaneously, via $(\hat Z_{s+r,\ell+r})_{r\geq 0}$.  When the
global traversals are not close the procedure will instead examine a
single strand from the ``lower'' global traversal, i.e.,
$(\hat Z_{s,\ell+r})_{0<r\leq n}$ if $\hat Z_{s,\ell}$ is positive,
and $(\hat Z_{s+r,\ell})_{0<r\leq n}$ if $\hat Z_{s,\ell}$ is
negative.

Fix $M\in\nwithoutzero$. To formally define the \emph{exploration
  procedure} $(\tilde Z_{k})_{k=0}^{Mn}$, set $\sigma_{0}=0$,
$\tilde Z_{0}=\hat Z_{0,0}$, and suppose that
$(\tilde Z_{k})_{0\leq k\leq \sigma_{j}}$ and
$\sigma_{0}<\dots< \sigma_{j}<Mn$ are given, with
$\tilde Z_{k}=\hat Z_{s(k),\ell(k)}$ for specified $s(k)$,
$\ell(k)$. In what follows we will write
$\tilde Z_{k}=\hat Z_{s,\ell}$, omitting that $s=s(k)$ and
$\ell=\ell(k)$. Then the exploration process continues as follows
while $\sigma_{j}<Mn$:
\begin{enumerate}
\item If $|\tilde Z_{\sigma_{j}}|<\eta(m)-1$ and $\tilde
  Z_{\sigma_{j}}=\hat Z_{s,\ell}$, then $\tilde
  Z_{\sigma_{j}+1}=\hat Z_{s,\ell+1}$ and $\tilde
  Z_{\sigma_{j}+2}=\hat Z_{s+1,\ell+1}$. Repeat this until
  \begin{equation}
    \label{eq:InnerExp}
    \sigma_{j+1} = \inf_{r>0}\Big\{\tilde Z_{\sigma_{j}+2r}\in
    \{\sqrt{m}-1,\sqrt{m},\eta(m)-1,\eta(m)\} \Big\}
  \end{equation}
  if this is at most $Mn$; otherwise set $\sigma_{j+1}=Mn$.
\item For $k\geq 0$, if $|\tilde Z_{\sigma_{j}+kn}| \geq \eta(m)-1$, then
  \begin{enumerate}
  \item if $\tilde Z_{\sigma_{j}+kn} = \hat Z_{s,\ell}>0$, then $\tilde
    Z_{\sigma_{j}+kn+r}=\hat Z_{s,\ell+r}$ for $0<r\leq n$.
  \item If $\tilde Z_{\sigma_{j}+kn} = \hat Z_{s,\ell}\leq 0$, then
    $\tilde Z_{\sigma_{j}+kn+r}=\hat Z_{s+r,\ell}$ for $0<r\leq n$.
  \end{enumerate}
  Let $k'$ be the minimal $k$ such that
  $|\tilde Z_{\sigma_{j}+kn}|<\eta(m)-1$. Set
  $\sigma_{j+1}=\sigma_{j}+k'n$ if this is at most $Mn$, and otherwise
  set $\sigma_{j+1}=Mn$.
\end{enumerate}
Thus $\tilde Z_{k}$ parametrizes $\hat Z_{s,\ell}$ in terms of the
weakly increasing sequences $s(k)$ and $\ell(k)$. Since $s$ and $\ell$
increase by at most one, this parametrization considers all
vertices in the evolution of $\V{x}$, $\V{y}$ under $\Phi$.

\begin{remark}
  \label{rem:cor}
  Strands of $\gt{\V{x}}$ are typically between two strands of
  $\gt{\V{y}}$. As $\gt{\V{x}}$ evolves, it may initially be closer to
  the lower strand of $\gt{\V{y}}$, and later be closer to the upper
  strand (and vice versa). The type (2) evolution accounts for this,
  and realizes the ``corridor'' structure present in \Cref{fig:displ}.
\end{remark}

To analyze the exploration process it will be convenient to subdivide
it into two types of \emph{process steps}, defined by a subsequence $t_{j}$
of the stopping times $\sigma_{j}$. Let $t_{0}=\sigma_{0}=0$,
$t_{1}=\sigma_{1}$, and, for $j\geq 1$,
\begin{equation*}
  t_{j+1} = Mn \wedge
  \begin{cases}
    \mathrm{argmin} \{ \sigma_{k}>t_{j} \mid |\tilde Z_{\sigma_{k}}|\in
    \{\eta(m)-1,\eta(m)\} & |\tilde Z_{t_{j}}| \leq \sqrt{m}
    \\
    \mathrm{argmin} \{ \sigma_{k}>t_{j} \mid |\tilde
    Z_{\sigma_{k}}|\leq \sqrt{m} 
    & |\tilde Z_{t_{j}}|\in
    \{\eta(m)-1,\eta(m)\}
  \end{cases}.
\end{equation*}
We think of $\ccb{t_{0},t_{1}-1}$ as an \emph{initial process step};
if $t_{1}=0$ this interval is the empty and there is no initial
process step. Subsequent intervals $\ccb{t_{i},t_{i+1}-1}$ are called
\emph{outer process steps} if
$|\tilde Z_{t_{i}}|\in \{\eta(m)-1,\eta(m)\}$, and \emph{inner process
  steps} otherwise.

We will primarily consider $\tilde Z_{k}$ up to the first time
$T=\inf \{2k \mid s(2k)\geq \Gamma n \text{ or } \ell(2k)\geq \Gamma
n\}$ that either $\gt{\V{x}}$ or $\gt{\V{y}}$ has been explored. Note
$T\leq 2\Gamma n$. Our first lemma rules out contacts during outer
process steps. Write $\gt{\V{x}}_{\ell}$ for the
$\ell$\textsuperscript{th} vertex in $\gt{\V{x}}$.
\begin{lemma}
  \label{lem:OutNoTouch}
  Suppose $\ccb{t_{i},t_{i+1}-1}$ is an outer process step of
  $(\tilde Z_{k})_{k\leq T}$. Then
  $d(\gt{\V{x}}_{s(t)},\gt{\V{y}}) \geq \sqrt{m}-1$ for all
  $t\in \ccb{t_{i},t_{i+1}-1}$ on $S(D)$.
\end{lemma}
\begin{proof}
  During an outer process step, both types of evolution (1) and (2) in
  the definition of the exploration process occur. The claim follows
  by the definition of an outer process step during an evolution of type
  (1). During evolution of type (2), strands are revealed one by one, and
  \Cref{lem:safezone} implies
  $d(\gt{\V{x}}_{s(t)},\gt{\V{y}})= \Theta(\sqrt{m\log m})$ for the
  $t$ considered while revealing a single strand.
\end{proof}

Combined with \Cref{lem:OutNoTouch}, the next lemma indicates that
$(\tilde Z_{k})_{k\leq T}$ is sufficient for detecting all contacts
between $\gt{\V{x}}$ and $\gt{\V{y}}$.

\begin{lemma}
  \label{lem:Close}
  Suppose $\tilde Z_{k} = \hat Z_{s,\ell}$, and that $S(D)$
  occurs. If $(\tilde Z_{k})_{k\leq T}$ evolves by {(1)}, then
  $|\tilde Z_{2k}|$ is the vertical distance
  $d(\gt{\V{x}}_{\ell},\gt{\V{y}})$. 
\end{lemma}
\begin{proof}
  On $S(D)$, there is a distance at least $D\sqrt{m\log m}$ between
  strands of $\gt{\V{y}}$. The lemma follows as
  $\eta(m)=\eta'\sqrt{m\log m}$, $\eta'=D/2$, and at even times during
  evolution by {(1)} $\tilde Z_{2k}$ is the vertical distance
  between two strands of $\gt{\V{x}}$ and $\gt{\V{y}}$.
\end{proof}

The next lemma reduces analyzing inner process steps to analyzing the ideal gap chain $Z$.
\begin{lemma}
  \label{lem:RIGC}
  There exists a constant $c>0$ such that the following holds on
  $S(D)$. Suppose that $\ccb{t_{j},t_{j+1}-1}$ is an inner process
  step of $(\tilde Z_{k})_{k\leq T}$ and $\tilde Z_{t_{j}}\neq
  0$. There is a coupling of $\tilde Z$ and the ideal gap chain
  during this process step such that
  $(|\tilde Z_{2k}|)_{2k\in \ccb{t_{j},t_{j+1}-1}}=(Z_{k})_{2k\in
    \ccb{t_{j},t_{j+1}-1}}$ with probability $1-e^{-c m^{1/2}}$.
\end{lemma}
\begin{proof}
  On $S(D)$, during an inner process step the previously explored
  strands are at distance at least $\Theta(\sqrt{m\log m})$ away. This
  includes the initial strand of $\gt{\V{x}}$ by \Cref{lem:GTnD}; this
  lemma applies since we consider $\tilde Z_{k}$ only up until $T$.
  Thus \Cref{lem:HC-CD} implies that computing transition
  probabilities of $\tilde Z$ by neglecting correlations from previous
  strands only introduces an exponentially small (in $\sqrt{m\log m}$)
  error. \Cref{lem:HC-CD} also implies that the error in computing
  these transition probabilities on $\mathbb{Z}$ instead of $C_{m}$ is
  also exponentially small (in $m$). The transition probabilities on
  $\mathbb{Z}$ are those that define the ideal gap chain, so the claim
  follows by a union bound, as $T\leq 2\Gamma n\leq m^{2}$.
\end{proof}

The following two remarks concern the disjointness of $\gt{\V{x}}$
and $\gt{\V{y}}$.
\begin{remark}
  \label{rem:success1}
  \Cref{lem:Close} implies that if $y-x<D\sqrt{m\log m}$, then
  $\gt{\V{x}}$ and $\gt{\V{y}}$ are disjoint on $S(D)$:
  $\V{y}\notin \gt{\V{x}}$ since consecutive
  strands of $\gt{\V{x}}$ are at least $D\sqrt{m\log m}$ apart on
  $S(D)$, and equation \eqref{eq:GTnD1} of \Cref{lem:GTnD} implies
  that subsequent strands of $\gt{\V{x}}$ do not get close to
  $\V{y}$. Similarly, \eqref{eq:GTnD1} implies that $\V{x}$ cannot be
  in $\gt{\V{y}}$. Thus $\gt{\V{x}}$ and $\gt{\V{y}}$ are disjoint, as
  every other vertex of $\gt{\V{y}}$ has an incoming arrow from a
  vertex not in $\gt{\V{x}}$. 
\end{remark}

\begin{remark}
  \label{rem:success2} 
  If $y-x>D\sqrt{m\log m}$, then $\V{y}\in \gt{\V{x}}$ is
  possible: this occurs if the exploration process starts its
  evolution according to (2) and hits zero. Note that this is
  consistent with \Cref{lem:OutNoTouch}, as the exploration process
  does not measure the vertical distance between strands when evolving
  according to (2). If $\V{y}\in\gt{\V{x}}$ we say the
  exploration process \emph{fails}, as the process was designed for
  exploring disjoint global traversals (continuing the exploration
  process re-uses already revealed arrows). Otherwise the exploration
  process \emph{succeeds}, and in this case each step of the
  exploration process reveals new arrows of~$\Phi$. Note that it
  cannot be that a subsequently revealed part of $\gt{\V{x}}$ is in
  $\gt{\V{y}}$ if $\V{y}\notin \gt{\V{x}}$.

  In what follows we only consider inner process steps when
  the exploration process succeeds, as in the case of failure the
  inner process degenerates to being identically zero.
\end{remark}

At this stage we can indicate how the proof of \Cref{thm:CofC}
will proceed. We wish to show concentration for the number of contacts
that occur during the inner process steps of a successful exploration
procedure. By \Cref{lem:RIGC} these inner process steps can simulated
by independent ideal gap chains. Concentration will follow provided
there are many inner process steps, and that during each of them the
number of contacts is reasonably concentrated. The first of these is
addressed in \Cref{sec:exc}, and the second in \Cref{sec:cons}.

\subsubsection{Analysis of Exploration Procedure Steps}
\label{sec:exc}

This subsection gives estimates for the duration of initial, inner,
and outer process steps. We also estimate the total number of process
steps during an exploration procedure. First we analyse the duration
of inner process steps.
\begin{lemma}
  \label{lem:innertime}
  Let $\tau$ denote the duration of an inner process step. There is a
  $c>0$ such that for $0<j\leq \sqrt{m}$ and $x>0$
  \begin{equation}
    \label{eq:IGC-Inner-Duration}
    \P\cb{\tau>x (m\log m) \,\Big|\, |\tilde Z_{0}|=j}\leq \exp(-cx).
  \end{equation}
  Moreover, $\E \tau = \Theta(m\log m)$.
\end{lemma}
\begin{proof}
  Note that $\tau$ is the hitting time of $\{\eta(m)-1,\eta(m)\}$ by
  $|\tilde Z|$. By \Cref{lem:RIGC} we can analyse $|\tilde Z|$
  during inner steps using the ideal gap chain $Z$.  To do this we
  will use the processes $Y^{\pm}$ introduced in the proof of
  \Cref{prop:idealgapestimates}; recall also the value $a$ that was
  introduced so that $Y^{\pm}$ have non-degenerate jump distributions
  when $Y^{\pm}\geq a$. We will first show
  \eqref{eq:IGC-Inner-Duration} using $\yminus$.

  Recall that $Z_{k}\geq \yminus_{2k}$. Let $\mathcal{E}_{t}$ be the
  event that $Z_{k}$ takes at least $c_{1}t$ non-zero jumps while
  $Z_{k}\geq a$ and $0\leq k\leq t$. Then there are $c_{1},c_{2}>0$
  such that $\P\cb{\mathcal{E}_{t}^{c}}\leq e^{-c_{2} t}$, as if
  $Z_{k}<a$, there is a uniformly positive probability that
  $Z_{k+a}>a$ by constructing an explicit path for the ideal gap
  chain.

  To establish \eqref{eq:IGC-Inner-Duration}, recall the modification
  $\yyminus$ of $\yminus$ that was introduced prior to
  \Cref{lem:BCouple}.  Then, given the above estimate on
  $\P\cb{\mathcal{E}_{t}}$, we have
  \begin{equation}
    \label{eq:innerupper}
    \P\cb{\tau^{Z}\geq x \,\Big|\, Z_{0}=j} \leq
    \P\cb{\tau^{\yminus}\geq 2c_{1}x \,\Big|\, \yyminus_{0}=a} + e^{-c_{2} x}
  \end{equation}
  by having $\yyminus$ evolve only when $Z\geq a$. The factor of $2$
  accounts for that $\yyminus$ takes two steps for every single step of
  $Z$. We have written, e.g., $\tau^{Z}$ for the hitting time of $Z$.

  To estimate $\P\cb{\tau^{\yyminus}\geq 2c_{1}x \,\Big|\,
    \yyminus_{0}=a}$, we use the Brownian coupling of
  Lemma~\ref{lem:BCouple}. Set $i=ym\log m$, $q=\beta$. Then
  \begin{align*}
   \P\cb{\tau^{\yyminus}>ym\log m \,\Big|\, \yyminus_{0}=j}
    &\leq
    \P\cb{ \Big\{\tau^{\yyminus}>ym\log m\Big\} \cap (E^{-}(i,q))^{c}
      \,\Big|\, \yyminus_{0}=j} +  \P\cb{E^{-}(i,q)}
    \\
    &\leq \P\cb{ \sup_{t\in [0,2\beta y m\log m]}|B(t)|\leq
      \kappa_{\eta(m)-1} \,\Big|\, B(0)=\kappa_{\sqrt{m}}} +
      Ce^{-c ym\log m} \\
    &\leq e^{-cy},
  \end{align*}
  where the second inequality is by~\eqref{e:Eprobs}, and the third
  has used that $\kappa_{\eta(m)-1} = \Theta( \sqrt{m\log m})$. Using
  this with $y=2c_{1}x$ and \eqref{eq:innerupper} with $x$ replaced by
  $xm\log m$ establishes \eqref{eq:IGC-Inner-Duration}.
  
  The estimate for $\E\tau$ follows by an analogous computation using
  $\yyplus$ to show that the probability $\P\cb{\tau > xm\log m \mid Z_{0}=j}$ is of
  order $\Theta(1)$
  if $x$ is small enough.
\end{proof}

We next consider the duration of outer process steps.
\begin{lemma}
  \label{lem:outertime}
  Let $\tau$ denote the duration of an outer process step. There is a
  $c>0$ such that
  \begin{equation}
    \label{eq:outertime}
    \P\cb{\tau>xm\log m\,\Big|\, S(D)} \leq e^{-cx}, \qquad 0\leq
    x\leq \sqrt{m}.
  \end{equation}
  Moreover, $\E \tau = \Theta(m\log m)$.
\end{lemma}
\begin{proof}
  We begin by showing \eqref{eq:outertime}. Note that on $S(D)$ an
  outer process step starts at distance at most $K\sqrt{m\log m}$,
  $K=4+\frac{3D}{2}$, as this is the largest distance between two
  consecutive strands. Hence it is enough to prove that there is a
  $c_{1}>0$ such that
  \begin{equation}
    \label{eq:outerphase}
    \P\cb{\tau> x m\log m \,\Big|\, |\tilde Z_{0}|=j} \leq \exp \{-c_{1}x\}, \qquad
    j\leq K\sqrt{m\log m}.
  \end{equation}
  During an outer process step we can replace the random variable
  $\phi_{\V{x}}$ of $\Phi$ by independent simple random walk
  increments with distribution $\phi$ at a cost of an error in
  probabilities of size at most $\exp\{-\Theta(m^{1/2})\}$ by
  \Cref{lem:OutNoTouch} and \Cref{lem:RW-CD}. We do so in what
  follows; this is the source of the upper bound on $x$ in
  \eqref{eq:outertime}. 

  Instead of determining $|\tilde Z_{j}|$ via the exploration
  procedure, we instead determine it by revealing $\gt{\V{y}}$, and
  then revealing $\gt{\V{x}}$.  Let $Z'_{j}$ be the distance from
  $\gt{\V{x}}_{j}$ to $\gt{\V{y}}$. By the preceding paragraph, $Z'$
  is a simple random walk with step distribution $\phi$ on an interval
  $\ccb{0,W}$, where $W\in \ccb{D\sqrt{m\log m},K\sqrt{m\log m}}$ is
  random. Let $\tau'$ be the first time at which $Z'$ is within
  distance $\sqrt{m}-1$ of the boundary of $\ccb{0,W}$. Then $\tau'$
  is stochastically dominated by the first time $Z'$ is within
  distance $\sqrt{m}-1$ of the boundary of $\ccb{0,K\sqrt{m\log m}}$.
  By a modification of the Brownian coupling from \Cref{lem:BCouple}
  (i.e., carrying out the same construction for simple random walk)
  and arguing as in the proof of \Cref{lem:innertime}, this implies
  $\P\cb{\tau'>x m\log m | Z'_{0}=j}\leq \exp\{-c_{1}x\}$ for all
  $j\in \ccb{0,K\sqrt{m\log m}}$. To conclude the proof of
  \eqref{eq:outertime}, observe that $\tau'$ can differ from $\tau$ by
  at most a multiplicative constant, as outer process steps must
  evolve both $\gt{V{x}}$ and $\gt{\V{y}}$ a positive proportion of
  the time.

  To obtain the order of $\E \tau$, it suffices to prove
  $\P\cb{\tau > xm\log m \mid S(D)} = \Theta(1)$ for some $x>0$. This
  follows from the discussion above as $Z'$ begins at distance
  $\Theta(\sqrt{m\log m})$ from the boundary of $\ccb{0,W}$ and
  $W\geq D\sqrt{m\log m}$.
\end{proof}

The next lemma gives a similar estimate for the duration of the
initial process step.
\begin{lemma}
  \label{lem:pretime}
  Let $\tau$ denote the duration of the initial process step started
  from $|\tilde Z_{0}|=j$, $K=4+\frac{3D}{2}$, and
  $k = (j/D\sqrt{m\log m} - K)_{+}$.
  There is a $c>0$ such that
  \begin{equation}
    \label{eq:pretime}
    \P\cb{\tau>nk+x m\log m \mid S(D)} \leq \exp(-cx), \qquad 0\leq x\leq \sqrt{m}.
  \end{equation}
\end{lemma}
\begin{proof}
  If $|\tilde Z_{0}|\leq K\sqrt{m\log m}$, then $k=0$ and this follows
  exactly from the arguments given for
  \Cref{lem:outertime,lem:innertime}. Now observe that since
  consecutive strands are separated by at least
  $D \sqrt{m\log m}$ on $S(D)$, if $j>K\sqrt{m\log m}$ then
  after (at most) $k$ strands, $|\tilde Z_{kn}|$ will be at most
  $K\sqrt{m\log m}$. We can then apply the argument for $k=0$.
\end{proof}

The exploration procedure alternates, after the initial process step,
between inner and outer process steps. We block together pairs of
inner and outer process steps into \emph{full process steps}. In the
next lemma the choices of $m^{3/8}$ and $m^{1/4}$ have been chosen for
convenience. 

\begin{lemma}
  \label{lem:ProcessSteps}
  Let $N$ denote the number of full process steps for a successful
  exploration process $(\tilde Z_{k})_{k\leq T}$.
  Suppose $|\tilde Z_{0}| =O(\sqrt{m^{3/2}\log m})$. There are
  $\alpha,c_{1}>0$ and $\mu=\mu(m)$ such that
  \begin{equation}
    \label{eq:fps}
    \P\cb{ |N - \mu| > \alpha^{-1}m^{3/8} \mid S(D)}\leq e^{-c_{1} m^{1/4}},
    \quad
    \mu = \Theta\ob{\frac{m^{3/2}}{(\log m)^{3/2}}}.
  \end{equation}
\end{lemma}
\begin{proof}
  Let $\tau_{i}$ denote the duration of the $i$\textsuperscript{th}
  process step, where the initial process step is indexed by $0$. Let
  $W_{i}=\tau_{2i-1}+\tau_{2i}$ be the duration of a full process
  step.

  We first consider what happens for an exploration process run up to
  time $Mn$. We assume $m^{7/16}\leq M\leq 2\Gamma$, and aim to show
  \eqref{eq:fps} with $\mu = Mn/(\alpha m \log m)$. Note that
  $N=k$ is equivalent to $Mn$ being in the interval
  $\ccb{\tau_{0}+\sum_{i=1}^{k}W_{i},
    \tau_{0}+\sum_{i=1}^{k+1}W_{i}-1}$. We can bound
  $\sum_{i=1}^{k}W_{i}$ above and below by an IID sum of
  sub-exponential random variables, as conditional on
  $\tilde Z_{\tau_{i}}$, $\tau_{i+1}$ is a sub-exponential random
  variable by \Cref{lem:outertime,lem:innertime,lem:pretime}. In the
  following we abuse notation slightly and continue to write $W_{i}$
  for these bounding variables.

  The rest of the proof consists of estimating the probability that
  $N=k$, i.e., that $W_{k+1}$ is the term that makes the sequence of
  partial sums $\sum_{i=1}^{r}W_{i}$ larger than
  $Mn-\tau_{0}\approx Mn$. More precisely, by
  \Cref{lem:outertime,lem:innertime,lem:pretime}, we have that
  \begin{equation}
    \label{eq:N1p}
    \P\cb{\tau_{0}\geq m^{11/8}\log m}\leq e^{-c m^{3/8}}, \quad
    \P\cb{W_{1}\geq m^{11/8}\log m}\leq e^{-c m^{3/8}}.
  \end{equation}
  In applying \Cref{lem:pretime} we have used the hypothesis $|\tilde
  Z_{0}|= O(\sqrt{m^{3/2}\log m})$. By \eqref{eq:N1p} it suffices to
  show that there is an $\alpha>0$ such that
  \begin{align}
    \label{eq:N2p}
    \P\cb{\sum_{i=1}^{k}W_{i}>Mn-2m^{11/8}\log m}
    &\leq e^{-cm^{1/4}}, \qquad k<\mu-\alpha^{-1}m^{3/8},
    \\
    \label{eq:N3P}
    \P\cb{\sum_{i=1}^{k}W_{i}<Mn}
    &\leq e^{-cm^{1/4}}, \qquad  k>\mu+\alpha^{-1}m^{3/8}, 
  \end{align}
  as the lemma then follows by a union bound over $k$ since
  $N\leq m^{2}$. Define $\alpha$ by $\alpha m\log m = \E W_{1}$;
  $\alpha = \Theta(1)$ by \Cref{lem:outertime,lem:innertime}. By
  Bernstein's inequality for subexponential random
  variables~\cite[Theorem~2.8.2]{Vershynin},
  \begin{equation*}
    \label{eq:N4p}
    \P\cb{ \Big| \sum_{i=1}^{k} W_{i} - k\alpha m\log m\Big| \geq t m\log m} \leq 2
    \max \{ e^{-c_{2}\frac{t^{2}}{k}}, e^{-c_{3}t}\}.
  \end{equation*}
  The claim follows by re-arranging \eqref{eq:N2p} and \eqref{eq:N3P}
  and applying this estimate. This uses
  $M\leq 2\Gamma \leq 2\sqrt{m}$, and for \eqref{eq:N2p} that
  $M\geq m^{7/16}$ so $m^{11/8}\log m = o(Mn)$. This completes the
  proof of the claim for an exploration procedure of length $Mn$.

  To finish the proof of the lemma, write $T=M_{T}n$. $T$ differs from
  $2\Gamma n$ only due to the phase of the initial process step that
  evolves a single global traversal. This phase takes time
  $O(m^{5/4})$ by the assumption
  $|\tilde Z_{0}|=O(\sqrt{m^{3/2}\log m})$; see the proof of
  \Cref{lem:pretime}. This alters $M_{T}$ by at most $m^{1/4}$, and
  hence $\mu$ by at most $o(m^{3/8})$. This completes the proof.
\end{proof}

\subsubsection{Consequences}
\label{sec:cons}

Write $L(t) = |\{ 0\leq s\leq t \mid Z_{s}=1\}|$ for the local time of the
ideal gap chain at state $1$ up to time $t$. Note that the law of
$L(t)$ depends on $Z_{0}$.
\begin{lemma}
  \label{lem:IGC-Inner-Contacts}
  Let $\tau$ be the hitting time of $\{\eta(m)-1,\eta(m)\}$ by the
  ideal gap chain with $Z_{0}=1$. Then 
  \begin{equation}
    \label{eq:IGC-IC}
    L(\tau) = \sum_{i=1}^{R} V_{i}
  \end{equation}
  where $R$ is a geometric random variable with success probability
  $\Theta( (\log m)^{-1/2})$ and the $V_{i}$ are IID geometric random
  variables with success probability $\Theta(m^{-1/2})$. The variables
  $R$ and $V_{i}$ are independent.
\end{lemma}
\begin{proof}
  Let $\tau_{1}<\infty$ be the first hitting time of
  $\{\sqrt{m}-1,\sqrt{m}\}$. For $r\geq 1$ let
  $h_{r} = \inf \{ j>\tau_{r} \mid Z_{j}=1\}$ and
  $\tau_{r+1}=\inf \{j>h_{r}\mid Z_{j}\in \{\sqrt{m}-1,\sqrt{m}\}$.
  The number $R-1$ of excursions to $1$ from $\{\sqrt{m}-1,\sqrt{m}\}$
  is a geometric random variable with success probability
  $\Theta( \frac{1}{\sqrt{\log m}})$ by
  \Cref{prop:idealgapestimates}. This is by applying the strong
  Markov property at time $\tau_{r}$, which yields an IID Bernoulli
  sequence of trials to hit $1$ before $\{\eta(m)-1,\eta(m)\}$; the
  variables are identically distributed as the hitting distribution on
  $\{\sqrt{m}-1,\sqrt{m}\}$ is IID after a visit to $1$. 

  If $h_{r}$ occurs, the number of visits $V_{r+1}$ to $1$ before
  returning to $\{\sqrt{m}-1,\sqrt{m}\}$ is a geometric random
  variable with success probability $\Theta(m^{-1/2})$ by
  \Cref{prop:idealgapestimates}. This proves the claim, as there are
  also $V_{1}$ visits to $1$ before hitting $\{\sqrt{m}-1,\sqrt{m}\}$.
\end{proof}

The next technical lemma will give control of the visits of the
exploration process to $1$.
\begin{lemma}
    \label{lem:SC}
    Let $N\in\nwithoutzero$, $\tilde N = \sum_{i=1}^{N}R_{i}$, and
    $\tilde W = \sum_{i=1}^{\tilde N}V_{i}$. Suppose the $V_{i}$ an
    IID sequence of sub-exponential random variables and the $R_{i}$
    an IID sequence of $\nwithzero$-valued sub-exponential random
    variables. There is a $c>0$ such that
    \begin{equation*}
      \label{eq:SC}
      \P\cb{ |\tilde W-\E\tilde W|>3N^{2/3}}\leq 4e^{-c N^{1/3}}.
    \end{equation*}
  \end{lemma}
  \begin{proof}
    By Bernstein's inequality~\cite[Theorem~2.8.1]{Vershynin}, there is a $c_{1}>0$
    such that
    \begin{equation}
      \label{eq:SC1}
      \P\cb{ \big|\tilde N - \E \tilde N\big|>N^{2/3}} \leq e^{-c_{1} N^{1/3}},
    \end{equation}
    so it suffices to estimate $|\tilde W - \E\tilde W|>3N^{2/3}$ on the
    event $\mathcal{E} = \{ |\tilde N-\E\tilde N|\leq
    N^{2/3}\}$. Write $\tilde W - \E\tilde W$ as
    \begin{equation}
      \label{eq:SC2}
      \sum_{i=1}^{\tilde N} V_{i}- \E V_{i} = \sum_{i=1}^{\E\tilde N}
      \ob{V_{i}- \E V_{i}} + {\bf 1}_{\tilde N>\E\tilde N}\sum_{i=\E \tilde
        N+1}^{\tilde N} \ob{V_{i} - \E V_{i}} - {\bf 1}_{\tilde N<\E\tilde N}
      \sum_{i=\tilde N+1}^{\E\tilde N}\ob{ V_{i}- \E V_{i}}.
    \end{equation}
    Bernstein's inequality implies there is a $c_{2}>0$ such that
    \begin{equation*}
      \label{eq:SC3}
      \P\cb{|\sum_{i=1}^{\E\tilde N}
      \ob{V_{i}- \E V_{i}}|>N^{2/3}}\leq e^{-c_{2}N^{1/3}},
  \end{equation*}
  and that (interpreting $\sum_{i=k+r}^{k}a_{i}=0$ if $r>0$)
  \begin{align*}
    \P\cb{\Bigg\{ \Big|\sum_{i=\E \tilde N+1}^{\tilde N} \ob{V_{i} - \E
    V_{i}}\Big| >N^{2/3}\Bigg\} \cap \mathcal{E}}
    &=
      \sum_{j=1}^{N^{2/3}} \P\cb{\tilde N = \E\tilde N + j,
      \Big|\sum_{i=1}^{j}\ob{V_{i} - \E V_{i}} \Big|>N^{2/3}} \\
    &\leq \max_{j\in\ccb{1,N^{2/3}}} \P\cb{\Big|\sum_{i=1}^{j}\ob{V_{i} -
      \E V_{i}} \Big|>N^{2/3}} \\
    &\leq e^{-c_{2}N^{1/3}}.
  \end{align*}
  A similar argument bounds the probability that
  $\sum_{i=\tilde N+1}^{\E \tilde N}(V_{i}-\E V_{i})>N^{2/3}$ on
  $\mathcal{E}$. The lemma follows by applying the triangle inequality
  to \eqref{eq:SC2}, using the above estimates, \eqref{eq:SC1}, and a union bound.
\end{proof}

\begin{proposition}
  \label{prop:CofT}
  Let $W$ be the number of visits of a successful exploration process
  $(\tilde Z_{k})_{k\leq T}$ to $1$. Suppose $|\tilde
  Z_{0}|=O(\sqrt{m^{3/2}\log m})$.  There is a $c>0$ such that
  \begin{equation}
    \label{eq:CofT}
    \P\cb{ \big|W-\nu \big|\geq m^{1/3} \mid S(D)}\leq 5e^{-cm^{1/7}}, \quad \nu =
    \Theta\ob{\frac{m}{\log m}}.
  \end{equation}
\end{proposition}
\begin{proof}
  Let $N$ denote the number of full process steps of $(\tilde
  Z_{k})_{k\leq T}$. By \Cref{lem:ProcessSteps} there is a $c_{1}>0$
  and $\mu=\mu(m)$ such that
  \begin{equation*}
    \P\cb{ \big|N-\mu\big|> \alpha^{-1}m^{3/8}} \leq e^{-c_{1}m^{1/4}}, \qquad \mu =
    \Theta\ob{\frac{m^{1/2}}{\alpha (\log m)^{3/2}}}.
  \end{equation*}
  Hence it suffices to verify \eqref{eq:CofT} on the event
  $|N-\mu|\leq \alpha^{-1}m^{3/8}$.

  By \Cref{lem:OutNoTouch} visits to one only occur during inner
  process steps. To count these visits, couple each inner process step
  of $\tilde Z$ with an independent copy of the ideal gap chain by
  \Cref{lem:RIGC}. This coupling may fail during the final inner
  process step, as the process $\tilde Z$ may end during this process
  step. We obtain a lower bound on $W$ by ignoring the final ideal gap
  chain contribution, and an upper bound by including it. By
  \Cref{lem:IGC-Inner-Contacts}, we have obtained 
  \begin{equation}
    \label{eq:Wbound}
    W\in \Big\llbracket\sum_{i=0}^{N} \sum_{j=1}^{R_{i}-1}V_{i,j}, \sum_{i=0}^{N+1}
      \sum_{j=1}^{R_{i}-1}V_{i,j}\Big\rrbracket,
  \end{equation}
  where the notation means the following. Let $Q$ be a geometric
  random variable with success probability $\Theta( (\log m)^{-1/2})$.
  In \eqref{eq:Wbound}, the $R_{i}$, $i\geq 1$ are identically
  distributed random variables, and conditionally on $R_{i}\geq 1$,
  $R_{i}$ is distributed as $Q$. The variable $R_{0}$ is
  stochastically bounded above by $Q$.  The $V_{i,j}$ are identically
  distributed geometric random variables with success probability
  $\Theta( m^{-1/2})$. All variables are independent. The upper index
  $R_{i}-1$ in the sums accounts for the fact that inner process steps
  do not start at $1$, unlike the hypotheses of
  \Cref{lem:IGC-Inner-Contacts}. The $R_{i}$ are IID as the initial
  distribution of $\tilde Z$ for an inner process step is always the
  same. The variable $R_{0}$ is singled out as it accounts for visits
  to $1$ in the initial process step; the domination claim holds as
  conditional on visiting $1$ in the initial process step, $R_{0}$ has
  the same distribution as $R_{1}$.

  The lemma now follows by a union bound and by applying \Cref{lem:SC}
  to both endpoints of the interval in \eqref{eq:Wbound}, using that
  we can work on the event $|N-\mu|\leq m^{3/8}$. The power $m^{1/7}$
  in the conclusion is for convenience, i.e., to drop logarithmic
  factors.  The order of $\nu$ is found by computing $\E W=\E N \E R \E V = \Theta(\frac{m^{1/2}}{\alpha (\log m)^{3/2}}
  \sqrt{\log m}\sqrt{m}) = \Theta(\frac{m}{\log m})$. 
\end{proof}

\begin{proof}[Proof of \Cref{thm:CofC}]
  We begin by giving the proof when $0<y-x<D\sqrt{m\log m}$. In this
  case, by \Cref{lem:Close}, the number of times $\gt{\V{x}}$ and
  $\gt{\V{y}}$ are at distance one is equal to the number of visits
  $W$ of $(\tilde Z_{k})_{k\leq T}$ to one, where $\tilde
  Z_{0}=y-x$. Let $A$ be the set of edges in a given column at which
  two global traversals are at distance one. On $S(D)$, the edges in
  $A$ are at distance $\Theta(\sqrt{m\log m})$ from one another. The
  arrows at the endpoints of such an edge form a contact with positive
  probability, and up to an exponentially small error this is true for
  any such set $A$ of contacts by \Cref{lem:RW-CD}. As a result we obtain that
  $\P\cb{ |X-\E{X}|\geq k^{2/3} \mid W=k, S(D)}\leq e^{-ck^{1/3}}$ by
  Bernstein's inequality. The theorem follows by \Cref{prop:CofT}.

  Next we consider the general case. The additional ingredient to
  handle is that since global traversals do not quite vertically span
  the torus, it may not be possible to follow all of $\gt{\V{x}}$ and
  $\gt{\V{y}}$ starting from two vertices at distance less than
  $D\sqrt{m\log m}$. Instead, let $\V{y}'$ be the closest point in
  $\gt{\V{y}}$ to $\V{x}$; by \Cref{lem:GTnD} \eqref{eq:GTnD1},
  $\V{y}'$ at distance $O(m^{3/4}(\log m)^{1/2})$ from
  $\V{x}$. Similarly there is a $\V{x}'\in \gt{\V{x}}$ within distance
  $O(\sqrt{m\log m})$ of $\V{y}$. To determine the number $X$ of
  contacts between $\gt{\V{x}}$ and $\gt{\V{y}}$ we run an exploration
  process started from (i) $\V{x}$ and $\V{y}'$ and (ii) from $\V{x}'$
  and $\V{y}$. In case (i) we run the process until $\gt{\V{y}}$ ends,
  and in case (ii) until $\gt{\V{x}}$ ends. The total time run by
  these exploration processes is
  $2\Gamma n - \Theta(m^{3/4}(\log m)^{1/2})$. The same argument as in
  the first paragraph relates $\tilde Z_{k}=1$ to contacts, and hence
  the claim now follows by applying \Cref{prop:CofT} (formally this
  lemma applies to a single exploration procedure, but the additional
  initial/terminal steps due to having two procedures do not alter the conclusion).
\end{proof}

\section{Dynamics on cycles}
\label{sec:dynamics}

\Cref{sec:HC-GD} introduced Glauber dynamics that leave the
equilibrium distribution of the hard-core model invariant. Since
$\Phi$ is equivalent to a collection of hard-core models, running the
Glauber dynamics on each column of $\T_{n,m}$ preserves $\Phi$ and its
cycle structure. \Cref{sec:GD-Bijections} describes the resulting
effective dynamics on the cycle structure, which imitate the random
transposition dynamics. This will allow us to prove \Cref{thm:PD1} in
\Cref{sec:cycles} by making use of methods developed by
Schramm~\cite{Schramm2011}.

Schramm's dynamical ideas are at this point relatively standard, and
have been used (and exposited) in other works concerning random
permutations, e.g.,~\cite{AKM,BKLM}. As such we provide detailed
proofs when our setting necessitates adjustments, but we have omitted
proofs when (up to notation) they are as given by Schramm.

\subsection{Glauber dynamics and the induced split-merge dynamics for
  $\Phi$} 
\label{sec:GD-Bijections}

\subsubsection{Dynamics}

The Glauber dynamics from \Cref{sec:HC-GD} naturally extend to give a
dynamics on the $n$ (disjoint) cycles $C_{m}$ that comprise the torus
$\T^{\star}_{n,m}$. To be precise, we define the \emph{directed
  spatial permutation Glauber dynamics} by repeatedly
\begin{enumerate}
\item choosing a uniformly random column $j\in \ccb{n}$, and
\item performing a Glauber update on the hard-core configuration on
  the $j$\textsuperscript{th} column of $\T^{\star}_{n,m}$.
\end{enumerate}
Updating the hard-core configuration may change the
corresponding arrow configuration, and hence the associated bijection
of $\T_{n,m}$. We record two basic facts about these dynamics.
\begin{lemma}
  \label{lem:Phi-dyn}
  Suppose $(\Phi_{i})_{i\geq 0}$ is the directed spatial permutations
  Glauber dynamics started from equilibrium, i.e., $\Phi_{0}=\Phi$
  conditioned to contain no global shift. Then:
  \begin{enumerate}
  \item $\Phi_{i}$ is distributed as $\Phi_{0}$ for all $i$, i.e., $\Phi_{0}$
    is the invariant measure of the dynamics.
  \item There is a $p_{0}>0$  (independent of $m$) such that
    $\P\cb{\Phi_{i+1}\neq\Phi_{i}}\geq p_{0}$ for all $i\geq 0$.
  \end{enumerate}
\end{lemma}
\begin{proof}
  \Cref{prop:HC-Rep} implies that $\Phi$ conditioned on having no
  global shift is equivalent to a product of hard-core models on the
  columns of $\T^{\star}_{n,m}$. The first claim follows as the
  Glauber dynamics leave the hard-core models on each column
  invariant. For {(2)}, it suffices to show that the directed spatial
  permutations Glauber dynamics have a uniformly positive probability
  to change the status of some vertex $\V{v}$ at which a contact
  occurs. In hard-core model language, these transitions (i) set
  $\V{v}$ to be unoccupied if it was previously occupied and (ii) set
  $\V{v}$ to be occupied if it and its neighbours were previously
  unoccupied. The claim follows as the density of vertices satisfying
  the conditions in (i) and (ii) is deterministically positive, and so
  are the transition probabilities.
\end{proof}

\subsubsection{Induced split-merge dynamics}
\label{sec:GD-SM}

This section describes the effective split-merge dynamics on cycles
that are induced by the directed spatial permutation Glauber
dynamics. Some results in this section make statements
about properties of global traversals without discussing the arbitrary
choices that must be made to decompose the cycles of $\Phi$ into
global traversals. The conclusions are well-defined as they hold for any such
decomposition; recall the discussion following \Cref{def:GTn}.

The next lemma constructs a good event on which we will be able to
analyse the effective dynamics. The error estimates are convenient
ones that suffice; we have not made an attempt to optimize
them. Recall that $C'$ is the constant in $C(m)=C'\sqrt{m\log m}$, and
recall that Theorem~\ref{thm:CofC} describes the concentration of the
number of contacts $X$ between two global traversals.
\begin{lemma}
  \label{lem:G}
  If $C'$ is sufficiently large, there is an event
  $\mathcal{G}\subset S(D)$ with $\P\cb{\mathcal{G}}=1-m^{-3}$ such
  that on the event $\mathcal{G}$:
  \begin{enumerate}
  \item \label{i:G1} There is a $c>0$ such that $\{\cycle\}_{\gtt} /
    \floor{\cycle}_{\gtt} < cm^{-1/8}$ for all cycles $\cycle$.
  \item \label{i:G2} For all cycles $\cycle_{i}$, $\cycle_{j}$, the number
      of contacts between them is
      \begin{equation*}
        \floor{\cycle_{i}}_{\gtt}(\floor{\cycle_{j}}_{\gtt}-1_{i=j})
        \E X (1+o(m^{-1/9})),
      \end{equation*}
      where $X$ denotes the number of contacts between two disjoint
      global traversals, and the implicit constant in $o(m^{-1/9})$
        is absolute.
  \item \label{i:G3} There is a $c_{2}>0$ such that each global traversal contains
    $c_{2}\Gamma n$ contacts.
  \end{enumerate}
\end{lemma}
\begin{proof}
  Take $D>6$. Choose $C'$ to be at least as large as indicated in
  \Cref{cor:All-Separated}, so that $\P\cb{S(D)}\geq 1-o(m^{-3})$ by
  \Cref{rem:Sdredef}. Then item \eqref{i:G1} holds on $S(D)$ by
  \Cref{lem:GTintegral} and \Cref{rem:Sdredef}. Let $\mathcal{E}_{3}$
  be the event that item \eqref{i:G3} holds. Then
  $\P\cb{\mathcal{E}_{3}}\geq 1- e^{-cm}$ by
  \Cref{lem:Strand-Contacts} and a union bound. We will construct
  $\mathcal{G}$ as sub-event of $S(D)\cap \mathcal{E}_{3}$ by
  intersecting the latter with a further event $\mathcal{E}_{2}$.

  Let $\mathcal{E}_{2} = \bigcap_{i,j}\mathcal{E}_{2}(i,j)$, where
  $\mathcal{E}_{2}(i,j)$ indicates that item \eqref{i:G2} holds for
  $\cycle_{i}$ and $\cycle_{j}$. It suffices to prove
  $\P\cb{\mathcal{E}_{2}(i,j)}$ occurs with all but stretched
  exponentially small probability. To show this, suppose $\cycle_{i}$
  and~$\cycle_{j}$ are comprised of $k_{i}$ and $k_{j}$ global
  traversals, respectively. On $S(D)$, $k_{i},k_{j}\geq 1$. The number
  of contacts between the global traversals of $\cycle_{i}$ and
  $\cycle_{j}$ is a sum of $k_{i}(k_{j}-1_{i=j})$ copies of~$X$,~$X$
  the number of contacts between two disjoint global traversals.  By
  \Cref{thm:CofC} and a union bound the number of contacts between the
  global traversals of $\cycle_{i}$ and $\cycle_{j}$ satisfies the
  desired estimate.

  What remains is to show that the contacts that involve fractional
  global traversals of $\cycle_{i}$ and $\cycle_{j}$ are taken into
  account by the $o(m^{-1/9})$ error term.  The number of contacts
  contained in the fractional global traversal of a cycle is at most
  twice the length of the fractional global traversal. On $S(D)$, item
  \eqref{i:G1} thus bounds above the number of contacts in the
  fractional global traversal of $\cycle_{i}$ by
  $2ck_{i} m^{-1/8}n = 2k_{i}\E X\cdot o( m^{-1/9})$.  Thus the number of
  contacts between $\cycle_{i}$ and $\cycle_{j}$ that are contained in
  fractional global traversals is at most
  $2(k_{i}+k_{j}) \E X \cdot o(m^{-1/9})$. The claim follows since
  $k_{i}+k_{j}\leq 2k_{i}k_{j}$.
\end{proof}

\begin{lemma}
  \label{lem:GTact}
  Conditional on $\mathcal{G}$, the probability that a uniformly
  chosen contact is contained in a fractional global traversal of a cycle
  $\cycle$ is at most $O(m^{-1/8}\floor{\cycle}_{\gtt} /
  \sum_{j}\floor{\cycle_{j}}_{\gtt})$. As a consequence, the probability that a
  uniformly chosen contact is between two global traversals is at least
  $1-O(m^{-1/8})$.
\end{lemma}
\begin{proof}
  Fix a decomposition of the cycles into global traversals. If there
  is a contact at $\V{v}\in\T^{\star}_{n,m}$, let
  $\{\V{v}_{1},\V{v}_{2}\} = \{ \V{v}\pm \frac12\}$. Let $\cycle_{1}$
  and $\cycle_{2}$ denote the cycles containing $v_{1}$ and $v_{2}$
  ($\cycle_{1}=\cycle_{2}$ may occur).

  Let $E_{i}$ be the event that $\V{v}_{i}$ is contained in the
  fractional part of $\cycle_{i}$. By \Cref{lem:G}\eqref{i:G1}, at
  most a $O(m^{-1/8})$ fraction of the strands of $\cycle_{i}$ are in
  the fractional part of $\cycle_{i}$, and the number of contacts in
  the fractional part of $\cycle_{i}$ is at most twice the length of
  the fractional part. Since the number of contacts in the global
  traversals of $\cycle_{i}$ is at least a constant fraction of the
  total length of the global traversals by \Cref{lem:G}\eqref{i:G3},
  the probability of $E_{i}$ is thus $O(m^{-1/8})$. The first
  conclusion follows. A union bound to estimate
  $\P\cb{E_{1}^{c}\cap E_{2}^{c}}$ then yields the second conclusion.
\end{proof}

\begin{lemma}
  \label{lem:CycleDynGT}
  Let $E_{i,j}$ be the event that a uniformly chosen contact is
  between $\cycle_{i}$ and $\cycle_{j}$. Let $\mathcal{E}_{\gtt}$ be
  the event that a uniformly selected contact is between two global
  traversals.  Then
  $\P\cb{E_{i,j}\mid \mathcal{E}_{\gtt}\cap \mathcal{G}}$ is
  proportional to
  $\floor{\cycle_{i}}_{\gtt} \floor{\cycle_{j}}_{\gtt}$ in the sense
  that
    \begin{equation*}
    \label{eq:CycleDynGT}
    \P\cb{E_{i,j}\mid \mathcal{E}_{\gtt}\cap \mathcal{G}} =
    \frac{\floor{\cycle_{i}}_{\gtt} (\floor{\cycle_{j}}_{\gtt}-1_{i=j})}
    {\sum_{k,\ell}\floor{\cycle_{k}}_{\gtt}   \floor{\cycle_{l}}_{\gtt}}(1+o(m^{-1/9})).
  \end{equation*}
  Moreover, in the case $i=j$ the contact is equally likely to be
  between any two distinct global traversals of $\cycle_{i}$, up to an
  $o(m^{-1/9})$ error. The implicit constants in $o(m^{-1/9})$
    are absolute.
\end{lemma}
\begin{proof}
  To estimate this probability we use that every contact is either
  between two cycles, or between a cycle and itself. The number of
  contacts between two cycles is given by
  \Cref{lem:G}\eqref{i:G2}. Thus the probability that a uniformly
  chosen contact that is contained in two global traversals is between
  $\cycle_{i}$ and $\cycle_{j}$ is (if $i\neq j$)
  $\frac{\floor{\cycle_{i}}_{\gtt} \floor{\cycle_{j}}_{\gtt}}
  {\sum_{k,\ell}\floor{\cycle_{k}}_{\gtt}
    \floor{\cycle_{\ell}}_{\gtt}}$ up to an error of size
  $o(m^{-1/9})$. If $i=j$ the numerator requires the subtraction of
  $1$ as no global traversal comes into contact with itself. (For the
  same reason, the denominator is in fact a slight overcount:
  $\floor{\cycle_{k}}_{\gtt}^{2}$ should be
  $\floor{\cycle_{k}}_{\gtt}(\floor{\cycle_{k}}_{\gtt}-1)$. This can
  be absorbed into the $o(m^{-1/9})$ error term since the total number
  of global traversals is of order $\sqrt{m\log m}$.)
\end{proof}

\subsection{Proof of Theorem~\ref{thm:PD1}}
\label{sec:cycles}

There are two steps in our proof, both using that the split-merge
dynamics induced by the Glauber dynamics is very similar to the random
transposition dynamics. The first step shows that most global
traversals in $\Phi$ are in fact contained in large cycles. The second
step, establishing the $\mathsf{PD}(1)$ limit, makes use of the fact
that we can perfectly couple the two dynamics for a (relatively) short
amount of time. This is sufficient as the dynamics acts very rapidly
on the large cycles of $\Phi$ --- this fact was one of the key
observations used in~\cite{Schramm2011} (see also~\cite{IoffeToth}).

\subsubsection{Existence of large cycles}
\label{sec:large}

The next two lemmas are analogues of Lemmas~2.3 and~2.4 of~\cite{Schramm2011}.

\begin{definition}
  The set of vertices in cycles $\cycle$ with
  $\floor{\cycle}_{\gtt}\geq j$ is denoted $\cycles(j)$.
\end{definition}
Note that on $\mathcal{G}$, $\cycles(1)$ is the set of all vertices,
as $\floor{\cycle}_{\gtt}\geq 1$. 

\begin{lemma}
  \label{lem:Sch2.3a}
  Let $\varepsilon\in\ob{0,\frac{1}{8}}$. Then
  \begin{equation}
    \label{eq:Sch2.3-1}
    \E \cb{ \abs{ \cycles(1)\setminus\cycles(\varepsilon C(m))}} \leq
    O(1) \varepsilon |\log \varepsilon| C(m).
  \end{equation}
\end{lemma}
Our proof of \Cref{lem:Sch2.3a} follows the same strategy Schramm used
to prove~\cite[Lemma~2.3]{Schramm2011} in the context of random
transpositions. There are two main differences. First, Schramm's lemma
was a non-equilibrium estimate, and some of the considerations
required are not present in our equilibrium setting. Second, we have
to take some care due to the fact that our control over the
split-merge dynamics is based on global traversals. Cycles are not
entirely comprised of global traversals, and this leads to errors when
comparing the dynamics to random transposition split-merge
dynamics. The proof shows these errors are negligible. Our exposition
of the proof follows that presented in~\cite{Schramm2011}, both to
facilitate comparison, and to emphasize that the key ideas are from
this reference.
\begin{proof}[Proof of \Cref{lem:Sch2.3a}]
  Let $(\Phi_{k})_{k\geq 0}$ denote the directed spatial permutation
  Glauber dynamics started from equilibrium, i.e., $\Phi$ conditional
  to contain no global shifts. We recall this is equivalent to $\Phi$
  itself as $m\to\infty$ by \Cref{lem:No-Shift}. The chain
  $(\Phi_{k})_{k\geq 0}$ is a lazy Markov chain, but by
  \Cref{lem:Phi-dyn}(2) there will be $k$ non-lazy steps with
  probability $1-o(1)$ after $Ck$ total steps provided $C$ is large
  enough. For this reason we measure time in terms of non-lazy
  steps in what follows.

  The proof considers intervals of time on which a typical cycle will
  double in size. Set 
  \begin{equation*}
    \label{eq:s2.3interval}
    \mu_{s}=\lceil a_{s}\rceil, \qquad a_{s}=2\frac{C(m)}{2^{s}}\log_{2}
    \frac{C(m)}{2^{s}}, \qquad s\geq 0,
  \end{equation*}
  $K=\lceil \log_{2}(\varepsilon C(m))\rceil$, and
  $\tau_{s}=\sum_{i=0}^{s-1}\mu_{i}$ for $s\geq 1$. Our argument will
  end with $s=K$; after this interval a cycle that has successfully
  doubled in size at each step has size $2^{K}\geq \varepsilon C(m)$. Set
  $t=\tau_{K}$.  We will show that \eqref{eq:Sch2.3-1} holds for
  $\Phi_{t}$ with $t = \Theta(C(m) \log C(m)) = \Theta( \sqrt{m (\log m)^{3}})$.  By
  stationarity this suffices. We begin with a few preparatory steps.

  Since $\mathbb{P}\cb{\mathcal{G}} = 1 - m^{-3}$ and each
  $\Phi_{k}$ is an equilibrium sample, the probability that
  $\mathcal{G}$ holds for $\Theta(t)$ steps is $1-o(m^{-1})$ by a
  union bound.  Thus by \Cref{lem:GTact,lem:CycleDynGT} we may use the
  following description of the dynamics. With probability at least
  $1-O(m^{-1/8})$ the dynamics is split-merge on the cycles, with
  their size measured in units of global traversals up to an error of
  size $m^{-1/9}$. With the complementary $O(m^{-1/8})$ probability
  either a cycle is split or two cycles are merged, but the cycles are
  selected with an unknown law and (if relevant) the distribution of
  the split is unknown. We call the first situation a \emph{good
    transposition} and the second a \emph{bad transposition}. A bad
  transposition corresponds to the event that a chosen contact is
  contained in a fractional global traversal. For the rest of the
  proof let $\gamma=1/10$, and assume $m$ is large enough such that
  both of these errors are at most~$m^{-\gamma}$.

  An important observation is that when a good tranposition occurs,
  the probability of splitting off a cycle of size $s$ can be bounded
  above by essentially the same argument as for the random
  transposition model on $\ccb{N}$ for $N\in\nwithoutzero$. For random
  transpositions, this probability is at most $2s/(N-1)$, as a random
  transposition is a choice of two random indices, and there are most
  $2s$ indices within distance $s$ of the first index. To formulate an
  upper bound for directed spatial permutations, let
  $N_{k}= \sum_{\cycle(k)}\floor{\cycle(k)}_{\gtt}$ denote the number
  of global traversals contained in the cycles of $\Phi_{k}$, i.e.,
  the sum is over the set of cycles of $\Phi_{k}$. At each step of the
  dynamics a split or merge occurs, $N_{k}$ may increase or decrease
  by one due to the joining of fractional global traversals /
  splitting of global traversals. Let $\mathsf{good}$ be the event
  that a good split occurs at time $k$.  Then
  \begin{equation*}
    \label{eq:SB-Split-true}
    \P\cb{\mathsf{good}, \text{ split $\leq 2s$}}
    \leq \frac{2s}{N_{k}-1}(1+o(m^{-\gamma})),
  \end{equation*}
  where split $\leq 2s$ is shorthand for the event that a cycle
  containing at most $2s$ global traversals is created. This gives a
  uniform (in $k$) upper bound since $N_{k}\geq C(m)/2$ on
  $\mathcal{G}$ and $\floor{\cycle}_{\gtt}\geq 1$ for all cycles.
  
  We now move to the main analysis, which proceeds by determining the
  probability that typical cycles fail to grow. There are two
  mechanisms to consider.  The proof will consider
  $|V|/(n\Gamma)$ for $V\subset V(\T_{n,m})$; this is an estimate of
  the number of global traversals in $V$. 
  \medskip
  
  \noindent
  \textbf{Failing due to a split}: Let $s\geq 0$ and
  $k\in \{\tau_{s}+1,\tau_{s}+2,\dots, \tau_{s+1}\}$. Let $F^{k}$ be
  the set of vertices contained in a cycle of size at most $2^{s+1}$
  that is created by a split at time $k$.  At most two such cycles can
  be produced by a single step of the dynamics. Write $\textsf{good}$
  for the event of a good transposition occurring and $\textsf{bad}$
  for the complementary event. Then
  \begin{align*}
    \label{eq:Fexp}
    \E \frac{|F^{k}|}{n\Gamma} &\leq 2\cdot (2^{s+1}+1)\ob{\P\cb{\textsf{good}, \text{ split
                 $<2^{s+1}$}}+\P\cb{\textsf{bad}, \text{ split $<2^{s+1}$}}} \\
               &\leq 2^{s+3}\ob{\frac{2^{s+2}}{N_{k}-1}(1+o(m^{-\gamma})) + o(m^{-\gamma})} \\
               &\leq \frac{2^{2s+7}}{C(m)-2} + 2^{s+3}m^{-\gamma}, 
  \end{align*}
  where the $+1$ in the first line accounts for the at most $n\Gamma$
  vertices that can be in a cycle but not in a global traversal.  Let
  $\tilde F^{k}=\cup_{\tau=1}^{k}F^{\tau}$. Then 
  \begin{equation*}
    \label{eq:Fexp2}
    \E \frac{|\tilde F^{t}|}{n\Gamma} \leq
    \sum_{s=0}^{K-1}\mu_{s}(\frac{2^{2s+7}}{C(m)-2} + 2^{s+3}m^{-\gamma}) =
    O(\varepsilon |\log\varepsilon| C(m)).
  \end{equation*}
  We have obtained the final estimate by using that the sum of the
  first term is $O(\varepsilon |\log\varepsilon|C(m))$ while the second is
  of order $o(C(m))$, as $m>C(m)$.
  \medskip

  \noindent
  \textbf{Failing to grow}: Let
  $k\in \{\tau_{s},\tau_{s}+1,\dots, \tau_{s+1}-1\}$, $s\geq
  0$. Cycles will not be likely to grow if there are too few
  reasonably sized cycles to merge with. To track this, set $H^{k}$ to
  be $\emptyset$ if $|\cycles^{k}(2^{s})|\geq \frac{C(m)}{2}$, and
  otherwise $H^{k}=\cycles(1)$. Set
  $\tilde H^{k}=\bigcup_{r=0}^{k}H^{r}$.
  
  $H^{k}=\cycles(1)$ is a failure event, but if $H^{k}=\emptyset$ a
  cycle may still fail to grow. Write $\cycles^{k}(s)$ for the set of
  vertices in cycles of size at least $s$ at time $k$. Define
 \begin{equation*}
   \label{eq:FtG1}
   B^{s}=\cycles^{\tau_{s}}(2^{s})\setminus \left(\tilde F^{\tau_{s+1}} \cup
     \tilde H^{\tau_{s+1}-1}\cup \cycles^{\tau_{s+1}}(2^{s+1})\right),
 \end{equation*}
 and $\tilde B^{s}=\bigcup_{r=0}^{s}B^{r}$. Observe
 \begin{equation}
   \label{eq:FtG2}
   \cycles(1)=\cycles^{t}(\varepsilon C(m))\cup \tilde H^{t}\cup
   \tilde F^{t}\cup \tilde B^{K-1}.
 \end{equation}
 If $v\in B^{s}$, then for $k\in \{\tau_{s},\dots, \tau_{s+1}-1\}$ we
 have $|\cycles^{k}(2^{s})|\geq \frac{C(m)}{2}$ since $B^{s}$ and
 $\tilde H^{\tau_{s+1}-1}$ are disjoint by definition.  Conditionally
 on $v$ being in a cycle containing $r\in \ccb{2^{s},2^{s+1}-1}$
 global traversals at time $k$, and on
 $|\cycles^{k}(2^{s})|\geq \frac{C(m)}{2}$, the probability the next
 step of the dynamics merges the cycle containing $v$ with another
 cycle containing at least $2^{s}$ global traversals is at least
 \begin{equation*}
   \label{eq:FtG3}
   2^{s}(\frac{C(m)}{2}-2^{s+1}){N_{k}\choose 2}^{-1}(1+o(m^{-\gamma})).
 \end{equation*}
 Since $N_{k}\leq C(m)$ and $\varepsilon\leq\frac18$, for $s\leq K-1$
 this is bounded below by 
 \begin{equation*}
   \label{eq:FtG4}
   2^{s}(\frac{C(m)}{2}-2\varepsilon C(m))\frac{2}{C(m)(C(m)-1)}(1+o(m^{-\gamma}))
   \geq \frac{2^{s-1}}{C(m)}(1+o(m^{-\gamma})). 
 \end{equation*}
 As a result,
  \begin{align*}
    \label{eq:FtG5}
    \P\cb{v\in B^{s}} \leq
    (1-\frac{2^{s-1}}{C(m)}(1+o(m^{-\gamma})))^{\mu_{s}}&\leq \exp
    \left(-2^{s-1}\mu_{s}\frac{(1+o(m^{-\gamma}))}{C(m)}\right) \\ &\leq
    O(\frac{2^{s}}{C(m)}) \exp \left(-o(m^{-\gamma})\log
      \frac{C(m)}{2^{s}}\right) \\ &= O(\frac{2^{s}}{C(m)})
  \end{align*}
  where the penultimate inequality used that $e>2$. By summing over
  $v$ this implies
  \begin{equation*}
    \label{eq:FtG6}
    \E \frac{|\tilde B^{K-1}|}{n\Gamma} 
    =O(\varepsilon C(m)).
  \end{equation*}

  \noindent
  \textbf{Concluding}: If $H^{k}=\cycles(1)$, then $|\tilde F^{k}\cup\tilde
  B^{s-1}|\geq \frac{n\Gamma C(m)}{2}$, so by Markov's inequality,
  \begin{equation}
    \label{eq:C1}
    \E\frac{|\tilde H^{t}|}{n\Gamma}\leq C(m)\P\cb{ |\tilde F^{t}\cup \tilde
      B^{K-1}|\geq \frac{n\Gamma C(m)}{2}}\leq 2\E\frac{|\tilde
      F^{t}\cup \tilde B^{K-1}|}{n\Gamma},
  \end{equation}
  and this expectation is of order
  $O(\varepsilon C(m)) + O(\varepsilon|\log\varepsilon|
  C(m))=O(\varepsilon|\log\varepsilon|C(m))$ by the preceding two
  sections of the proof. Combined with~\eqref{eq:FtG2} we are done.
\end{proof}

\begin{lemma}
  \label{lem:Sch2.4}
  Let $0< \epsilon,\alpha<\frac{1}{8}$, and let $M$ be the minimal
  number of cycles that contain a $(1-\epsilon)$-fraction of the
  vertices of $\T_{n,m}$. Then there are $m_{1}(\epsilon)$ and a
  universal constant $K$ such that for all $m\geq m_{1}(\epsilon)$
  \begin{equation}
    \label{eq:Sch2.4}
    \P\cb{M>\alpha^{-1}|\log\alpha\epsilon|^{2}}\leq K \alpha.
  \end{equation}
\end{lemma}
\begin{proof}
  Lemma~\ref{lem:Sch2.3a} states
  $\E \cb{ |\cycles(1)\setminus \cycles(\epsilon C(m))|}\leq O(1)
  \epsilon |\log\epsilon |C(m). $ This equation implies the result
  exactly as discussed in the final paragraph of~\cite[Proof of
  Lemma~2.4]{Schramm2011}.
\end{proof}

\subsubsection{Existence of macroscopic cycles and convergence to
  $\mathsf{PD}(1)$}
\label{sec:PD}

Write $\Phi^{(m)}$ to emphasize the $m$-dependence of $\Phi$. To
establish \Cref{thm:PD1} we must show that if $Y$ has distribution
$\mathsf{PD}(1)$, then for any $\epsilon>0$, if $m$ is sufficiently
large then there is a coupling of $\cstructure(\Phi^{(m)})$ and $Y$
such that
$\P\cb{ \norm{\cstructure(\Phi^{(m)})-Y}_{\infty}<\epsilon}\geq
1-\epsilon$.  A theorem due to Schramm will be
useful~\cite{Schramm2011}.
\begin{theorem}
  \label{thm:Sdiscrete}
  Suppose $(\pi^{(N)}_{k})_{k\geq 0}$ is the random transposition chain on
  $\ccb{N}$ with an initial condition $\pi^{(N)}_{0}$. Suppose that
  for $\epsilon'>0$ small enough,
  \begin{equation}
    \label{eq:Sdiscrete}
    \E\cb{ N-|\cycles(\epsilon' N)|} < O(1) \epsilon' |\log \epsilon'| N
  \end{equation}
  where the expectation is with respect to $\pi_{0}$.  Let $Y$ have
  distribution $\mathsf{PD}(1)$. For $\epsilon>0$ let
  $q=\epsilon^{-1/2}$. If $N$ is large enough there is a
  coupling of the random transposition dynamics such that
  \begin{equation*}
    \P\cb{ \norm{\cstructure(\pi^{(N)}_{q})-Y}_{\infty}<\epsilon}\geq
1-\epsilon. 
  \end{equation*}
\end{theorem}
\begin{remark}
  \Cref{thm:Sdiscrete} is not formally stated
  in~\cite{Schramm2011}. We have extracted the hypotheses as discussed
  in~\cite[Section~4]{Schramm2011}, where Schramm proves his main
  theorem.
\end{remark}

We can now complete our proof of \Cref{thm:PD1}.
\begin{proof}[Proof of \Cref{thm:PD1}]
  Fix $\epsilon>0$ and let $q=\epsilon^{-1/2}$. Given $\Phi$, let $N$
  denote the number of global traversals contained in the cycles of
  $\Phi$. By \Cref{lem:G}\eqref{i:G1}, $N\geq m/(2\Gamma)$ on the
  event $\mathcal{G}$, as each cycle contains at least one global
  traversal, and the total amount of the system in fractional global
  traversals cannot exceed that in global traversals. Note that (see
  \Cref{def:GTn}) $m/\Gamma\to\infty$ as $m\to\infty$.

  Let $\pi_{0}$ be the partition of $\ccb{N}$ that agrees with the
  cycle structure of $\Phi$. That is, enumerate the global traversals,
  and if two global traversals are in a common cycle, we put them in a
  common part.  Arbitrarily choose a permutation compatible with the
  given block structure. By \Cref{lem:Sch2.3a,lem:Sch2.4} and the
  first paragraph, $\pi_{0}$ satisfies the hypotheses of
  Theorem~\ref{thm:Sdiscrete} when $m$ is sufficiently large.

  To complete the proof of the theorem we will use a coupling of the
  split-merge dynamics on the cycles of $\Phi$ induced by the directed
  spatial permutation Glauber dynamics with the split-merge dynamics
  induced by the random transposition dynamics started from $\pi_{0}$.
  Write $\Phi_{t}$ and $\pi_{t}$ for the respective marginals. The
  coupling has to account for the fact that the number of global
  traversals $N_{t}$ in $\Phi_{t}$ can change with $t$, and goes as
  follows. Note that $N_{t+1}-N_{t}\in \{-1,0,1\}$. By
  \Cref{lem:G,lem:CycleDynGT} the induced split-merge dynamics selects
  two distinct uniformly chosen global traversals with probability
  $1-o(m^{-1/9})$, and we can couple this perfectly with choosing two
  distinct vertices under the random transposition dynamics if
  $N_{t}=N_{0}$. To account for the possibility that
  $N_{0}\neq N_{t}$, suppose $N_{t}-N_{0}>0$ and mark $N_{t}-N_{0}$ of
  the global traversals. The probability that one of these global
  traversals is selected by the dynamics is at most $2t/N_{0}$; if
  neither is selected we can couple a step of the dynamics as
  above. If $N_{t}-N_{0}<0$ we couple similarly, but this time instead
  adding $|N_{t}-N_{0}|$ fictitious global traversals. The probability
  that our coupling succeeds without ever selecting a marked or
  fictitious global traversal in the first $q$ steps is of order
  $1-\Theta(q^{2})/N_{0} = 1-o(1)$. The probability of always
  selecting two distinct global traversals is of order
  $1-qm^{-1/9}=1-o(1)$. Hence we may assume each step is coupled as
  above.

  Under this coupling the number of global traversals in a cycle of
  $\Phi_{t}$ may not be exactly equal to the number of indices in the
  corresponding part of $\pi_{t}$, but the discrepancy is at most
  $t$. The discrepancy can arise due to merging or splitting of
  fractional global traversals. Since $\Phi$ is invariant under the
  Glauber dynamics, this implies that the size
  $\floor{\cycle_{i}}_{\gtt}$ of the $i$\textsuperscript{th} largest
  cycle of $\Phi$ is within $q$ of the number of vertices contained in
  the $i$\textsuperscript{th} largest cycle of $\pi_{q}$. Since $q$ is
  of order one as $m\to\infty$, this discrepancy is negligible due to
  the normalization present in the definition of the cycle
  structure. This shows that we can couple the vector
  $(\floor{\cycle_{i}}_{\gtt} /
  (\sum_{j}\floor{\cycle_{j}}_{\gtt}))_{i\geq 1}$ with a
  $\mathsf{PD(1)}$ sample as desired by \Cref{thm:Sdiscrete}.  To
  obtain the statement of \Cref{thm:PD1}, which measures cycle sizes
  by the number of vertices they contain, we use \eqref{eq:SysSize},
  which implies that measuring size by a count of global traversals is
  equivalent (up to errors of size $o(1)$) to measuring by a count
  of vertices.
\end{proof}

\subsubsection*{Acknowledgements}
\label{sec:acknowledgements}

A.H.\ is supported by the National Science Foundation under DMS grants
1855550 and 2153359 and by the Simons Foundation as a 2021 Simons
Fellow. T.H.\ was supported by an NSERC PDF when this work was
initiated. We thank the referee for a careful
  reading of the paper, and for helpful comments and suggestions.

\bibliographystyle{plain}
\bibliography{DSP}

\end{document}